\documentclass[10pt,oneside,reqno]{amsart}

  \usepackage[utf8]{inputenc}
  \usepackage[T1]{fontenc}
  \usepackage{lmodern}

  \usepackage{hyperref}

  \usepackage{amsmath}
  \usepackage{amsthm}
  \usepackage{amsfonts}
  \usepackage{amssymb}
  \usepackage{amscd}
  \usepackage{amsbsy}

  \usepackage{pdfpages}
  \usepackage{fancyhdr}
  \usepackage{geometry}
  \usepackage{setspace}
  \usepackage{xcolor}
  \usepackage{pifont}

  \usepackage{graphicx}
  \usepackage{tabularx}
  \usepackage{epsfig}
  \usepackage[all]{xy}
  \usepackage{tikz}
  \usetikzlibrary{matrix}

  \usepackage{fixmath}

  \usepackage{appendix}

  \usepackage{enumerate}

  \usepackage[sort, numbers]{natbib}
  \usepackage{bibentry}

  \newtheorem{theorem}{Theorem}[section]

  \newtheorem*{ack}{Acknowledgments}
  
  \newtheorem{lemma}[theorem]{Lemma}

  \newtheorem{corollary}[theorem]{Corollary}
  
  \theoremstyle{definition}
  \newtheorem{definition}[theorem]{Definition}

  \newtheorem*{remark}{Remark}

  \numberwithin{equation}{section}

  \newcommand{\N}{{\mathbb N}}
  \newcommand{\Z}{{\mathbb Z}}

  \newcommand{\R}{{\mathbb R}}
  \newcommand{\C}{{\mathbb C}}
  
  \newcommand{\T}{{\mathbb T}}

  \newcommand{\leb}{{\operatorname{Leb}}}

  \newcommand{\dd}{{\mathrm d}}

  \newcommand{\tr}{\operatorname{tr}}

  \renewcommand{\Re}{\operatorname{Re}}

  \newcommand{\dom}{\text{\rm{Dom}}}

  \allowdisplaybreaks

\title{Ballistic Transport for Limit-periodic Schr\"odinger Operators in One Dimension}

\author[G. Young]{Giorgio Young}
\address{Department of Mathematics, Rice University, Houston, TX~77005, USA}
\email{gfy@rice.edu}
\thanks{G.Y.\ was supported in part by NSF grant DMS--1745670.}

\begin{document}
\begin{abstract}
	In this paper, we consider the transport properties of the class of limit-periodic continuum Schr\"odinger operators whose potentials are approximated exponentially quickly by a sequence of periodic functions. For such an operator $H$, and $X_H(t)$ the Heisenberg evolution of the position operator, we show the limit of $\frac{1}{t}X_H(t)\psi$ as $t\to\infty$ exists and is nonzero for $\psi\ne 0$ belonging to a dense subspace of initial states which are sufficiently regular and of suitably rapid decay. This is viewed as a particularly strong form of ballistic transport, and this is the first time it has been proven in a continuum almost periodic non-periodic setting. In particular, this statement implies that for the initial states considered, the second moment grows quadratically in time. 
\end{abstract}

\maketitle

\section{Introduction}
	For a bounded function $V:\R \to \R$, we consider the Schr\"odinger operator $H$ defined by the expression 
	\[H:=-\frac{\dd^2 }{\dd x^2}+V\] 
	with domain $\dom(H)=H^2(\R)$, where $H^2(\R)$ is the Sobolev space of twice weakly differentiable functions with second derivative in $L^2(\R)$. $H$ is a bounded-operator perturbation of $H_0:=-\frac{\dd^2}{\dd x^2}$, $\dom(H_0)=H^2(\R)$,
and is thus an unbounded self-adjoint operator on $L^2(\R)$.

  In this paper, we study the quantum evolution corresponding to $H$ described by the Schr\"odinger equation:
  \begin{align}\label{eq:Schrod}
  i\frac{\partial \psi}{\partial t}=H\psi,\quad \psi(0)=\psi\in H^2(\R).
  \end{align}
   Since $H$ is self adjoint, one may use the spectral theorem to define $e^{-itH}$, which forms a strongly continuous one-parameter unitary group, and find the solution $e^{-itH}\psi$ to \eqref{eq:Schrod}. From this expression and the associated machinery, one can describe the dynamics of the equation.

We will be particularly interested in the transport properties of solutions to \eqref{eq:Schrod} for $H$ in a particular class of limit-periodic operators. To be more precise, we will need to introduce some objects used to describe transport. The position operator $X$ is the  unbounded operator defined as
  \begin{align*}
    (X\psi)(x)=x\psi(x),\quad   \dom(X):=\{ \psi\in L^2(\R): x\psi\in L^2(\R)\},
  \end{align*}
	and for a Schr\"odinger operator $H$, the Heisenberg evolution of the position operator $X$ is defined as
 \[X_H(t):=e^{itH}Xe^{-itH}.\]
  For $p>0$ and suitably localized states $\varphi$, the moments are defined by
  \[
  |X|_\varphi^p(t):=\||X|^{p/2}e^{-itH}\varphi \|^2=\int_\R|x|^{p}|(e^{-itH}\varphi)(x)|^2\dd x,
  \]
and the transport exponents are 
  \begin{align*}
	  \beta_\varphi^+(p):=\limsup_{t\to \infty}\frac{\log|X|_\varphi^p(t)}{p\log t},\;\;\beta_\varphi^-(p):=\liminf_{t		\to \infty}\frac{\log|X|_\varphi^p(t)}{p\log t}, 
  \end{align*}
 measuring the growth of the moments in time on a power law scale. 
 
  In the continuum setting, ballistic transport is often described as quadratic growth of the second moment for a dense set of suitably localized and regular initial states.\footnote{Ballistic upper bounds $\beta_{\varphi}^+(2)\leq 1$ are available in great generality in the continuum setting \cite{RadinSimon}.} For the operators we consider, this will follow from our main theorem. Namely, we find a $Q\psi\in L^2(\R)$ with
 \begin{align}\label{eq:10920}
\lim_{t\to\infty}\frac1tX_H(t)\psi=Q\psi,\qquad Q\psi\ne 0,
\end{align}
for $\psi\ne 0$ in a subset of $L^2(\R)$ which contains the Schwartz functions.
This condition is the strong form of ballistic transport referenced in the abstract; indeed, from \eqref{eq:10920}, one has
\begin{align*}
\lim_{t\to\infty}\frac1{t^2}|X|_{\psi}^2(t)=\|Q\psi\|^2>0,
\end{align*}
which in turn implies $\beta_{\psi}^\pm(2)=1$ for these $\psi$.
 
There is a relationship between transport properties and the spectral theory for Schr\"odinger operators. 
 More qualitatively, the connection between spectral type and transport is seen through the RAGE theorem \cite{Damanik,ReedSimon3}. More quantitative results illustrating this connection are also available. For example, it is known that Schr\"odinger operators with pure point spectrum satisfy 
 \begin{align}\label{eq:10922}
 \lim_{t\to\infty}\frac1{t^2}|X|_\varphi^2(t)=0
 \end{align}
for all $\varphi$ of compact support \cite{Simon2}, which is described there as the absence of ballistic motion. There are also the works of Guarneri, Combes and Last \cite{Guarneri1,Guarneri2,Combes,Last}. In particular, the Guarneri-Combes-Last theorem \cite{Last} implies that, in one dimension and in the presence of absolutely continuous spectrum, there is transport in a quantitative sense. Defining the spectral measure corresponding to a state $\varphi\in L^2(\R)$ and a self adjoint operator $H$ to be the unique positive finite Borel measure $\mu_{\varphi}$ on $\R$ such that for all $z\in\C\setminus \R$,
  \begin{align*}
    \langle \varphi,(H-z)^{-1}\varphi\rangle=\int_{\R}\frac{1}{x-z}\dd\mu_{\varphi}(x),
  \end{align*}
  the Guarneri-Combes-Last theorem implies that in one dimension, and for states $\varphi$ whose spectral measures have a nonzero absolutely continuous part, an averaged ballistic lower bound holds: there exists a $\varphi$ dependent constant $C$ such that  
  \begin{align}\label{eq:0909}
	\frac1T\int_0^T|X|_{\varphi}^2(t)\dd t\geq CT^2.
  \end{align}
Since $\beta^+(2)$ is bounded from below by its time averaged analog, $\mu_\varphi$ having a nonzero absolutely continuous part is enough to imply $\beta_{\varphi}^+(2)=1$. However, it is not yet known whether in one dimension $\mu_\varphi$ having nonzero absolutely continuous part is enough to ensure $\beta_{\varphi}^-(2)=1$. We note this is known to be false in higher dimensions \cite{BellSchulz,KiselevLast}, where work often seeks to show averaged ballistic lower bounds of the form \eqref{eq:0909}, for example  \cite{Stolzetall} in the two dimensional almost periodic setting. Typically, showing the quantitative characterizations of transport beyond the scope of the Guarneri-Combes-Last theorem requires model dependent methods. To better contextualize our results, we will focus our introductory discussion on non-time averaged results in one-dimensional, almost periodic models. 

 Given the relationship described above and the diverse spectral properties of almost periodic operators \cite{Simon}, it is perhaps not surprising that the transport properties for these operators is similarly varied. In \cite{AschKnauf,DLY}, strong ballistic transport is shown for periodic continuum Schr\"odinger operators and (block) Jacobi matrices respectively, both of which are well known to have absolutely continuous spectrum. Specifically, in \cite{AschKnauf}, \eqref{eq:10920} is shown for a dense subspace of suitably regular and localized states $\psi$. Meanwhile in \cite{DLY}, the corresponding notion of strong ballistic transport in the discrete setting is shown to hold. Namely, they prove strong resolvent convergence of the operators $\frac1tX_H(t)$ to an operator with trivial kernel, and use this to show $\beta_\varphi^\pm(p)=1$ for $p>0$ and another dense set of $\varphi$. 
 
 On the other hand, in \cite{BGS, Bourgain,Sinai,FSW, AJ, BG}, certain quasi-periodic operators are shown to exhibit pure point spectrum, and in fact Anderson localization. As we noted above, pure point spectrum implies the absence of ballistic motion as described by \eqref{eq:10922}. See also the results \cite{BJ} which imply strong dynamical localization for the almost Mathieu operator under a certain regime. In the discrete limit-periodic setting, work of Damanik and Gorodetski \cite{DamanikGorodetski} found a set of limit-periodic operators with pure-point spectrum which is dense in the space of all limit-periodic operators. Further work of P\"oschel \cite{Poschel} and Damanik and Gan \cite{DamanikGan1,DamanikGan2} gives examples of limit-periodic operators which exhibit an extremely strong form of Anderson localization; in particular, it is known that this notion of localization implies strong dynamical localization.

 However, ballistic transport is often hoped for in classes of one dimensional almost periodic operators that are known to have absolutely continuous spectrum. There have been a number of recent results showing ballistic transport for such operators, which we summarize briefly. The paper \cite{Fillman} is the main inspiration for our work. In that paper, Fillman shows that limit-periodic Jacobi matrices which are exponentially quickly approximated exhibit strong ballistic transport. There have also been results in the setting of quasi-periodic operators, particularly those with ``small,'' analytic potentials, which often have purely absolutely continuous spectrum. In Kachkovskiy's work in the discrete setting \cite{Kachkovskiy}, what is described there as ``frequency-averaged strong ballistic transport'' is shown for a large class of these operators; in particular, this implies strong ballistic transport along a subsequence of time scales and for almost every frequency. The works of Zhao \cite{Zhao,Zhaocont}, in both the discrete and continuous setting show that for diophantine frequencies and small enough potentials, the second moment grows quadratically, a weaker notion of transport than that of \cite{Kachkovskiy} which, however, holds without passing to a subsequence of times and a full measure set.  
 
 Subsequently, a paper of Zhang and Zhao \cite{ZhaoZhang} explicitly linked the values of the transport exponents and absolutely continuous spectrum in the setting of discrete one frequency quasi-periodic operators. They find $\beta_\varphi^\pm(p)=1$ for suitably localized states $\varphi\ne 0$ and almost every frequency whenever the operator has purely absolutely continuous spectrum. Finally, there is the recent work of Kachkovskiy and Ge \cite{KachkovskiyGe}, which proves a form of ballistic transport not quite as strong as the strong ballistic transport referenced above, but sufficient to deduce that the second moment grows at least quadratically. Their results hold for a broad class of discrete quasi-periodic operators.

 This paper contributes to the limit-periodic literature. We find strong ballistic transport for a class of continuum limit-periodic operators known to have absolutely continuous spectrum. These results may be viewed as an extension of the main result of \cite{Fillman} to the continuum.  To introduce our results, we define the class of limit-periodic operators considered. A limit-periodic operator is one with a potential that is a limit-periodic function, i.e. there is a sequence of continuous $p_n$-periodic functions $V_n$ with
  \begin{align*}
  \lim_{n\to \infty}\|V-V_n\|_\infty=0,
  \end{align*}
  where $\| V\|_\infty$ denotes the essential supremum of the function $V$.
  It is a standard fact for limit-periodic functions that the potentials can be chosen so that $p_{n+1}/p_n\in \N$.  We define those operators whose potentials are approximated exponentially quickly by the above sequence, or are in ``exponential class $\eta$,'' as follows.
  \begin{definition}
  Let $\eta>0$. A limit-periodic Schr\"odinger operator $H$ is said to be of exponential class $\eta$ if there is a sequence of continuous $p_n$-periodic functions $V_n$ such that
  \[
  \lim_{n\to \infty}e^{\eta p_{n+1}}\|V_n-V\|_\infty=0
  \]
  where $p_{n+1}/p_{n}\in \N$ and $p_n\ne p_{n+1}$ for any $n\in \N$.
  We say $H\in EC(\infty)$ if $H$ is of exponential class $\eta$ for any $\eta>0$.
  \end{definition}

The class $EC(\infty)$ is well-studied, beginning with work of  Pastur and Tkachenko \cite{PasturTkachenko1} and Chulaevskii \cite{Chulaevskii}. Despite the varied spectral properties of limit-periodic operators generally, this class does behave similarly to periodic operators in important ways. In particular, like periodic operators, members of this class are known to have absolutely continuous spectrum \cite{Chulaevskii, AvronSimon,PasturTkachenko1,PasturTkachenko2, PasFig}. However, the spectrum of operators in $EC(\infty)$ is also generically nowhere dense and perfect, i.e. a Cantor type set \cite{Chulaevskii, AvronSimon,PasturTkachenko1,PasturTkachenko2,PasFig,Simon}. As we have noted, the former of these properties is of particular interest given the thesis of this paper.

We are now ready to state our results precisely. We will make use of the following subspace of $L^2$, defined for a given $s>0$ as 
\begin{align*}
\mathcal{D}_s:=\{\psi\in H^2(\R)\cap \dom(|X|^s): \psi''\in \dom(|X|^s)\}
\end{align*}  
where $\dom(|X|^s):=\{ \psi\in L^2(\R):|x|^s\psi\in L^2(\R)\}$. The subspace $\mathcal{D}_s$ will allow us to optimize certain statements in terms of the regularity and decay of the initial states considered. Of course, $\mathcal{D}_s$ contains the Schwartz functions, and is dense in $L^2(\R)$. We prove the following.

\begin{theorem}\label{thm:main}
For each $R>0$, there is a constant $\eta=\eta(R)$ such that if $H\in EC(\eta)$ 	and $\|V\|_\infty\leq R$, then for each $\psi\in \mathcal{D}_s$ with $s>2$, there is a $Q\psi\in L^2(\R)$ so that \eqref{eq:10920} holds. In particular, if $H\in EC(\infty)$, \eqref{eq:10920}, holds for $\psi\in\mathcal{D}_s$, $s>2$.
  \end{theorem}

	Before proceeding to the proofs of the above, we make some comments on the methodology and organization of the paper. The general strategy is based on Fillman's work in the discrete setting \cite{Fillman}. However, in addition to the technical difficulties that arise naturally when working with unbounded operators, new arguments and techniques are required to deal with new obstacles in key places. For example, as in \cite{Fillman}, finding a quantitative version of the convergence result that yields strong ballistic transport for periodic operators is critical to our work; this is our Theorem~\ref{thm:Quantper}. The proof of this theorem requires an estimate on a normed difference given by an integral over quasimomenta $k$ in the Brillouin zone. This estimate is found by first partitioning the Brillouin zone into two time dependent sets: one ``good'' set of $k$ on which we have decay in time for a quantity that appears in the integrand, and a ``bad'' set of $k$ that must have Lebesgue measure that decays in time at an explicit rate, and then finding an estimate on the integrand that is uniform in both $k$ and the operator. In the discrete setting, this uniform estimate is found using the Hilbert-Schmidt norm of the operators involved, which are finite matrices for each $k$. In the continuum, the analogous operators are unbounded, and this estimate is significantly more delicate. The estimate on the measure of the ``bad'' set in \cite{Fillman} takes advantage of an explicit finite product formula for a derivative. While in the continuum setting, there is a similar product formula, it is generally over infinitely many terms. To circumvent difficulties that arise in suitably controlling this product, our proof is by different methods.
	
	Section 2 contains some estimates on the Hill discriminant of a periodic operator, culminating in the estimate on the Lebesgue measure of the ``bad'' set of quasimomenta described above. The quantitative version of the convergence result of \cite{AschKnauf} is proved in Section 3. The control given by this theorem will allow us to extract the $Q\psi$ of Theorem~\ref{thm:main} for a limit-periodic operator as a limit of the associated quantity for the periodic approximants, and show that $Q\psi\ne0$; this argument is found in Section 4, along with some propagation estimates it requires. The appendix contains estimates on the expectation of powers of the position operator with careful control of the constants appearing in the bounds.
	
	 \begin{ack}
  We would like to thank the anonymous referee for their careful reading of this manuscript. Their insightful comments led to a modification in the presentation of the Floquet transform, and subsequent simplifications in the proofs in the third section of this paper. 
  \end{ack}
	
\section{Periodic operators and Hill discriminant estimates}

  We introduce some of the Floquet-Bloch theory for periodic operators, for which \cite{ReedSimon4,Kuchment,AvronSimon} all serve as helpful references. Fundamental to this theory is the Floquet transform and direct integral decomposition. We take $\Gamma^*=\frac{2\pi}{p}\Z$ to be the dual lattice corresponding to the lattice $\Gamma=p\Z$, and define the tori 
\begin{align*}
  \T=\R/\Gamma, \quad \T^*=\R/\Gamma^*.
  \end{align*}
  We will refer to $k\in \T^*$ as the quasimomentum, and when needed, we will fix the Brillouin zone $\mathcal{B}:=\left(-\frac{\pi}{p},\frac{\pi}{p} \right]$ as a fundamental domain for $\Gamma^*$. We define the constant fiber direct integral
  \[
  L^2\left(\T^*,\frac{\dd k}{|\T^*|}; L^2(\T,\dd q)\right)=\left\{ f:\T^*\to L^2(\T,\dd q); \int_{\T^*}\|f\|_{L^2(\T,\dd q)}^2\frac{\dd k}{|\T^*|}<\infty\right\}
  \]
  as well as the Floquet transform:
  \begin{align*}
  U:&L^2(\R)\to L^2\left(\T^*,\frac{\dd k}{|\T^*|}; \;L^2(\T,\dd q)\right)=:\int_{\T^*}^\oplus L^2(\T,\dd q)\frac{\dd k}{|\T^*|},
  \end{align*}
  given by the $\int_{\T^*}^\oplus L^2(\T,\dd q)\frac{\dd k}{|\T^*|}$ convergent series
  \begin{align}\label{eq:fourier}
  U\psi(k,q)=\sum_{\ell\in\Z}e^{-ik(q+p\ell)}\psi(q+p\ell).
  \end{align}
   The Floquet transform is  unitary, with inverse $U^{-1}=U^*$ given by
   \begin{align}
   U^*f(q+p\ell)=\int_{\T^*}e^{ik(q+p\ell)}f(k,q)\frac{\dd k}{|\T^*|}
   \end{align}
   for $\ell\in \Z$. As our notation suggests, $\int_{\T^*}^\oplus L^2(\T,\dd q)\frac{\dd k}{|\T^*|}$ is canonically isomorphic to $L^2(\T^*\times \T)$. 

  For a $p$-periodic $V\in L^\infty(\R)$, the Floquet transform conjugates the Schr\"odinger operator $H=-\frac{\dd^2}{\dd x^2}+V$ with domain $H^2(\R)$ to a direct integral of the operators 
  \begin{align*}
  H(k)=(D+k)^2+V,\quad \dom(H(k))=H^2(\T)
  \end{align*} 
  where $D=-i\frac{\dd}{\dd q}$. Symbolically,
  \begin{align}
    UHU^*=\int_{\T^*}^{\oplus}H(k)\frac{\dd k}{|\T^*|}.
  \end{align}
  Each of the $H(k)$ is a bounded-operator perturbation of the operator 
\begin{align*}  
  H_0(k):=(D+k)^2,\quad\dom(H_0(k))= H^2(\T), 
  \end{align*}
  and thus has a compact resolvent. Let $E_n(k)$ be the eigenvalues of $H(k)$ listed in ascending order,
  \[E_n(k)\leq E_{n+1}(k),\;k\in\T^*,n\in\N,\]
  with strict inequality for $k\not\in\frac{\pi}{p}\Z$ when the eigenvalues are simple, see, for example \cite{AvronSimon}.
  
  Then, for $k\not\in \frac{\pi}{p}\Z$, $H(k)$ admits the eigenfunction expansion
  \begin{align*}
    H(k)=\sum_{n=1}^\infty E_n(k)P_n(k),
  \end{align*}
  where $P_n(k)$ are rank-one orthogonal projections onto the associated eigenspaces.\footnote{There is an eigenfunction expansion for all $k$, but due to possible degeneracy in the $E_n(k)$ for $k\in \frac{\pi}{p}\Z$, the projections could be of rank-two at these points, requiring a re-indexing for these $k$. Since our statements involving the expansion concern integrals over $\T^*$, we may avoid the issue by introducing the formula away from these $k$.} For $k\in \mathcal{B}$, $H(k)$ and $H(-k)$ are antiunitarily equivalent, so that $E_n(k)=E_n(-k)$. $E_n(k)$ is analytic on $\left(0,\frac{\pi}{p}\right)$ and continuous at the endpoints, with $(-1)^{n+1}\frac{\dd E_n}{\dd k}>0$ on this interval.

  Using the theory of direct integrals and the properties of the $E_n(k)$ outlined above, the spectrum of $H$ may be seen to have a ``band" structure:
  \begin{align*}
    \sigma(H)=\bigcup_{j=0}^\infty[\lambda_{2j},\lambda_{2j+1}]
  \end{align*}
  where $\lambda_{2j}<\lambda_{2j+1}$ and the intervals $[\lambda_{2j},\lambda_{2j+1}]$ are referred to as bands. These bands are parametrized by the eigenvalues $E_n(k)$:
  \begin{align*}
    \begin{cases}
      \lambda_{2j}=E_{j+1}(0),\;\lambda_{2j+1}=E_{j+1}\left(\frac{\pi}{p} \right)& j\;\text{even}\\
      \lambda_{2j}=E_{j+1}\left(\frac{\pi}{p}\right),\;\lambda_{2j+1}=E_{j+1}\left(0 \right)& j\;\text{odd}.
    \end{cases}
  \end{align*}

  In the one dimensional setting, we have access to the discriminant, an important tool in the theory. We introduce the operator 
  \[
  \tilde H(k)=-\frac{\dd^2}{\dd x^2}+V,\quad \dom(\tilde H(k))=\{ \psi\in H^2([0,p]): \; \psi(0)=e^{-ikp}\psi(p), \; \psi'(0)=e^{-ikp}\psi'(p)\}.
  \]
We note that conjugating $H(k)$ by the multiplication operator $\psi(\cdot)\mapsto e^{-ik\cdot}\psi(\cdot)$ yields $\tilde H(k)$, so that the two operators are unitarily equivalent. Let $u_1(z,x)$ and $u_2(z,x)$ the Neumann and Dirichlet solutions to the differential equation $-u''+Vu=zu$ associated to $\tilde H(k)$; $u_1'(z,0)=u_2(z,0)=0$ and $u_1(z,0)=u_2'(z,0)=1$. Then, the monodromy matrix is defined by
  \begin{align*}
    M(z)=\begin{pmatrix}
      u_1(z,p)&u_2(z,p)\\
      u_1'(z,p)&u_2'(z,p)
    \end{pmatrix}
  \end{align*}
  and the discriminant is defined as
  \begin{align*}
    \Delta(z)=\tr M(z)=u_1(z,p)+u_2'(z,p).
  \end{align*}
  Since $M(z)$ is the one step transfer matrix for the differential equation associated to $\tilde H(k)$, $E$ is an eigenvalue for $\tilde H(k)$, and so $H(k)$, if and only if $e^{ikp}$ is an eigenvalue of the matrix $M(E)$. Since $\det(M(z))=1$, this yields
  \begin{align}\label{eq:dk}
    \Delta(E_n(k))=2\cos(pk),\;k\in\left[0,\frac{\pi}{p}\right].
  \end{align}

Our results in this section will rely on an upper bound on the derivative of the discriminant on the spectrum, established in \cite{FillmanLukic}[Lemma 2.1]. As noted in that paper, such an estimate is likely well known, although we also could not find a reference. Since we require a cruder estimate with more assumptions on the potential than they do, we record the version of their statement which suffices for our work.

  \begin{lemma}[\cite{FillmanLukic}, Lemma 2.1]\label{lem:FillLuk}
Let $V$ a $p$ periodic potential with $\| V\|_\infty\leq R$. Then the corresponding discriminant $\Delta$ satisfies
  \begin{align}\label{eq:discriminant}
  |\Delta'(\lambda)|\leq C_2p^2\exp(2C_2R^{1/2}p)
  \end{align}
  for all $\lambda\in \sigma(H)$ and where $C_2$ is a constant independent of $p$, $V$, and $R$.
  \end{lemma}

  We will also need the following lemma, consisting of two estimates. The first provides a lower bound on the group velocity in a band, and the second is a closely related lower bound on the difference of an eigenvalue $E_n$ evaluated at different quasimomenta. Since both estimates will be useful in what follows, we state them separately.

  \begin{lemma}\label{lem:velocitybound}
  Let $V$ a $p$ periodic potential with $\| V\|_\infty\leq R$. For $k\in \mathcal{B}$ we have
  \begin{align}\label{eq:velocitybound}
  \left|\frac{\dd E_m}{\dd k}\right|\geq \frac{2|\sin(pk)|}{C_2p}\exp\left(-2C_2R^{1/2}p. \right)
  \end{align}
  and for $k_1,k_2\in \left[0,\frac\pi{p}\right]$,
  \begin{align}\label{eq:Holderestimate}
  |E_m(k_2)-E_m(k_1)|\geq \frac{\exp\left(-2C_2R^{1/2}p\right)}{\pi C_2}(k_2-k_1)^2.
  \end{align}
  \end{lemma}

  \begin{proof}
   Differentiating \eqref{eq:dk} in $k$ yields
  \begin{align*}
  \frac{\dd E_m}{\dd k} =\frac{-2p\sin(pk)}{\Delta'(E_m(k))}
  \end{align*}
  and \eqref{eq:velocitybound} follows from the symmetry $E_m(k)=E_m(-k)$ and Lemma~\ref{lem:FillLuk}.

  To prove \eqref{eq:Holderestimate}, we assume without a loss of generality that $0\leq k_1\leq k_2\leq \frac{\pi}{p}$. Thus, since $(-1)^{m+1}\frac{\dd E_m}{\dd k}\geq 0$ for $0\leq k\leq \frac{\pi}{p}$,
  \begin{align*}
  |E_m(k_2)-E_m(k_1)|&=(-1)^{m+1}\int_{k_1}^{k_2}\frac{\dd E_m}{\dd k}\dd k\\
  &=\int_{k_1}^{k_2}(-1)^m\frac{2p\sin(pk)}{\Delta'(E_m(k))}\dd k\\
  &=\int_{k_1}^{k_2}\frac{2p\sin(pk)}{|\Delta'(E_m(k))|}\dd k
  \end{align*}
  since $(-1)^m\Delta'(E_m(k))\geq 0$. Thus, by Lemma~\ref{lem:FillLuk},
  \begin{align*}
    |E_m(k_2)-E_m(k_1)|&\geq \frac{\exp\left(-2C_2R^{1/2}p\right)}{C_2p^2}\int_{k_1}^{k_2}2p\sin(pk)\dd k.
  \end{align*}

  We now show that for $u_1,u_2\in[0,\pi]$, $\int_{u_1}^{u_2}\sin(u)\dd u\geq \frac1{2\pi}(u_2-u_1)^2$. For $u_1,u_2\in \left[0,\frac{\pi}{2}\right]$
  we have $\int_{u_1}^{u_2}\sin(u)\dd u\geq \frac1\pi(u_2-u_1)^2$ by integrating the inequality $\sin(u)\geq \frac{2}{\pi}u$; the same bound then holds for $u_1,u_2\in \left[\frac{\pi}{2}, \pi\right]$ by symmetry. Thus, for $0\leq u_1\leq \frac{\pi}{2}\leq  u_2\leq \pi$, we have
  \begin{align*}
  \int_{u_1}^{u_2}\sin(u)\dd u&\geq \frac{1}{\pi}\left(\frac{\pi}{2}-u_1 \right)^2+\frac1\pi\left(\frac{\pi}{2}-u_2\right)^2\geq \frac{1}{2\pi}(u_2-u_1)^2
  \end{align*}
  where the final inequality can be seen by setting $u_1=x+\frac{\pi}{2}$ and $u_2=y+\frac{\pi}{2}$, for $x,y\in \R$ and using $\frac{1}{2}(x+y)^2\geq 0$.

  Thus, by the change of variables $pk\mapsto k$,
  \begin{align*}
  2\int_{k_1}^{k_2}\sin(pk)p\dd k\geq \frac{p^2}{\pi}(k_2-k_1)^2,
\end{align*}
  and \eqref{eq:Holderestimate} follows.
  \end{proof}

  \begin{remark}
    Exploiting the relationship $\rho(E)=\left|\frac{\dd E_n}{\dd k}\right|^{-1}$, where $\rho(E)$ is the Radon-Nikodym derivative of the density of states measure with respect to Lebesgue measure, and $E=E_n(k)$ \cite{AvronSimon}, \eqref{eq:Holderestimate} provides a H\"older estimate for the density of states measure. Such estimates are mentioned in the literature \cite{AvronSimon,Moser}, however, critically for this paper, the estimate above provides a H\"older constant with explicit dependence on the period.
  \end{remark}

  The latter of these estimates allows us to find a bound on the measure of the set of quasimomenta for which the eigenvalues corresponding to different bands can be nearby. This will be used to prove the main result of the next section. 

  \begin{lemma}\label{lem:DOSCont}
  Define the set
  \[
  \mathcal B_\epsilon:=\{ k\in\mathcal{B}:\exists n,m\in \N,\; n\ne m\;\text{s.t.}\;\;|E_m(k)-E_n(k)|\leq \epsilon\}.
  \]
  
  Then, 
  \begin{align}\label{eq:DOSCont}
  \leb(\mathcal B_\epsilon)\leq 4\sqrt{C_2\pi}\exp\left(C_2R^{1/2}p\right)\cdot \sqrt{\epsilon}.
  \end{align}
  \end{lemma}

  \begin{proof}
  Since for each $m$, and $k\in \left(0,\frac{\pi}{p} \right)$, $E_m(k)=E_m(-k)$, excluding $k=\pi/p$, $\mathcal B_\epsilon$ is symmetric, and
  \[
  \leb(\mathcal B_\epsilon)=2\leb\left(\mathcal B_\epsilon\cap \left[0,\frac{\pi}{p}\right] \right).
  \] 
 Using that bands do not overlap, we find
  \[
  \mathcal B_\epsilon\cap\left[0,\frac{\pi}{p}\right]\subseteq\bigcup_{\ell=1}^\infty B_\ell
  \]
  where we define 
  \[
  B_\ell:=\left\{k\in \left[0,\frac{\pi}{p}\right]:|E_\ell(k)-E_\ell(0)|\leq \epsilon\right\}\cup \left\{k\in \left[0,\frac{\pi}{p}\right]:\left|E_\ell(k)-E_\ell\left(\frac{\pi}{p}\right)\right|\leq \epsilon\right\}.
  \]

  By \eqref{eq:Holderestimate}, we have
    \[
  B_\ell\subseteq \left(\left[0,\sqrt{C_2\pi}\exp\left(C_2R^{1/2}p\right)\sqrt{\epsilon}\right]\cup \left[\frac{\pi}{p}-\sqrt{C_2\pi}\exp\left(C_2R^{1/2}p\right)\sqrt{\epsilon},\frac{\pi}{p}\right]\right)\cap \left[0,\frac{\pi}{p} \right]
  \]
  
  for each $\ell\in \N$.
Thus,
  \begin{align}
  \leb\left(B_\epsilon\cap\left[0,\frac{\pi}{p}\right]\right)\leq2\sqrt{C_2\pi}\exp\left(C_2R^{1/2}p\right)\cdot \sqrt{\epsilon},
  \end{align}
  from which \eqref{eq:DOSCont} follows.
  \end{proof}

\section{An estimate on the rate of convergence of $\frac{1}{t}X_H(t)\psi$ for periodic $H$}

  In this section, $H$ will denote a $p$ periodic Schr\"odinger operator with bounded potential on $L^2(\R)$. Theorem 2.3 of \cite{AschKnauf} shows that for these operators and $\psi\in H^1(\R)\cap \dom(X)$,
  \begin{align*}
    \lim_{t\to\infty}\frac1tX_H(t)\psi&=U^*\left(\int_{\T^*}^\oplus\sum_{n=1}^\infty P_n(k)(D+k)P_n(k)\frac{\dd k}{|\T^*|} \right)U\psi\\
    &=
    U^*\left(\int_{\T^*}^\oplus\sum_{n=1}^\infty \frac{\dd E_n(k)}{\dd k}P_n(k)\frac{\dd k}{|\T^*|}\right)U\psi.
  \end{align*}
  Critically for our purposes in this paper, the proof in \cite{AschKnauf} relies on uniform boundedness of the difference and uses density of truncated eigenfunction expansions to show convergence, and thus does not provide a quantitative estimate on the speed of convergence. The main result of this section is then to provide such an estimate.
  We will denote
  \begin{align}\label{eq:2023}
  Q_H\psi=U^*\left(\int_{\T^*}^\oplus\sum_{n=1}^\infty \frac{\dd E_n(k)}{\dd k}P_n(k)\frac{\dd k}{|\T^*|}\right)U\psi.
  \end{align}

  We establish a power law in time upper bound on the difference $\frac1tX_H(t)\psi-Q_H\psi$. Since we will be examining sequences of periodic operators with increasing periods and seeking information on their limit, it is important that we record precisely how the constants that appear in this estimate depend on the period. This bound will follow as a corollary of the below theorem. We denote by $D(r)=e^{irH}De^{-irH}$ the Heisenberg evolved momentum operator, and we remind the reader that we have defined $\mathcal{D}_s:=\{\psi\in H^2(\R)\cap \dom(|X|^s): \psi''\in \dom(|X|^s)\}$.

  \begin{theorem}\label{thm:Quantper}
  Let $R>0$, $s>2$  and $p\geq \pi$. If $H$ is a $p$-periodic Schr\"{o}dinger operator with $\|V\|_\infty\leq R$, and $\psi\in \mathcal{D}_s$, then
  \begin{align}\label{eq:strongestimate}
  \left\|\frac12Q_H\psi-\frac1t\int_0^tD(r)\psi\dd r \right\|\leq C_1M_3p^{3/2}\exp\left(\frac12C_2R^{1/2}p\right)t^{-1/5}
  \end{align}
  where  $C_1=C_1(R)$, and $M_3=M_3(\psi)$ are constants.
  \end{theorem}

  \begin{remark}
  We make the assumption $p\geq \pi$ here and in some of the statements that follow because it simplifies some of our estimates, here the right hand side of \eqref{eq:strongestimate}. A similar bound holds for small $p$. Since we will be examining a sequence of potentials with periods tending to infinity, this assumption poses no issues.
  \end{remark}

  Using the equality \cite{AschKnauf,RadinSimon} for $\psi\in H^1(\R)\cap \dom(X)$
  \begin{align}\label{eq:integral}
  X_H(t)\psi=X\psi+2\int_0^tD(r)\psi\dd r
  \end{align}
  the following corollary follows immediately.

  \begin{corollary}\label{cor:Corollary}
   Let $R>0$, $s>2$  and $p\geq \pi$. If $H$ is a $p$-periodic Schr\"{o}dinger operator with $\|V\|_\infty\leq R$, and $\psi\in \mathcal{D}_s$, then
  \begin{align}
   \left\|Q_H\psi-\frac1tX_H(t)\psi \right\|\leq t^{-1}\|X\psi\|+2C_1M_3p^{3/2}\exp\left(\frac12C_2R^{1/2}p\right)t^{-1/5}
  \end{align}
    where  $C_1=C_1(R)$, and $M_3=M_3(\psi)$ are constants.
  \end{corollary}

  We will need a pair of estimates uniform in the quasimomentum and independent of the potential $V$ with $\|V\|_\infty\leq R$, the first of which is below. 

  \begin{lemma}\label{lem:relbound}
  Let $R>0$ and suppose $\|V\|_{\infty}\leq R$.  Then, the operator
  \begin{align*}
    (H(k)+2R)^{1/2}(D+k)(H(k)+2R)^{-1}:L^2(\T,\dd q)\to L^2(\T,\dd q)
  \end{align*}
  is bounded uniformly in $k$: there is an $M_1=M_1(R)$ such that
  \begin{align}\label{eq:0414}
  \|(H(k)+2R)^{1/2}(D+k)(H(k)+2R)^{-1}\|^2\leq M_1.
  \end{align}
  \end{lemma}

  \begin{proof}
  For $\varphi\in H^2(\T)$, we have by self-adjointness
  \begin{align*}
  \|(H(k)+2R)^{1/2}\varphi\|^2&=\langle \varphi,(H(k)+2R)\varphi\rangle\leq \langle \varphi,(H_0(k)+3R)\varphi\rangle =\|H_0(k)^{1/2}\varphi\|^2+3R\|\varphi\|^2.
  \end{align*}
  For $\varphi\in H^1(\T)$, we take $\varphi_n\in H^2(\T)$, with $\| \varphi_n-\varphi\|_{H^1(\T)}\to 0$. Then, by the above bound, the sequence $(H(k)+2R)^{1/2}\varphi_n$ is Cauchy in $L^2(\T)$. Since $(H(k)+2R)^{1/2}$ with domain $H^1(\T)$ is closed, $(H(k)+2R)^{1/2}\varphi_n\to (H(k)+2R)^{1/2}\varphi$. Thus, we may conclude the same bound for $\varphi\in H^1(\T)$,
  \begin{align}\label{eq:1june22}
  \|(H(k)+2R)^{1/2}\varphi\|^2&\leq \|H_0(k)^{1/2}\varphi\|^2+3R\| \varphi\|^2.
  \end{align}

  Using \eqref{eq:1june22} with $\varphi=(D+k)(H(k)+2R)^{-1}\psi$ for $\psi\in L^2(\T)$ we have
  \begin{align*}
  \|(H(k)+2R)^{1/2}&H_0(k)^{1/2}(H(k)+2R)^{-1}\psi\|^2\\
  &\leq \| H_0(k)^{1/2}(D+k)(H(k)+2R)^{-1}\psi\|^2+3R\|(D+k)(H(k)+2R)^{-1}\psi\|^2\\
  &=\| H_0(k)(H(k)+2R)^{-1}\psi\|^2+3R\|H_0(k)^{1/2}(H(k)+2R)^{-1}\psi\|^2.
  \end{align*}
  Using the triangle inequality on the first term yields
  \begin{align*}
    \| H_0(k)(H(k)+2R)^{-1}\psi\|&\leq \|\psi\|+\| (V+2R)(H(k)+2R)^{-1}\psi\|\\
    &\leq \|\psi\|+ 3R\|(H(k)+2R)^{-1}\psi\|.
  \end{align*}
  For the second, $(H(k)+2R)^{-1}:L^2(\T)\to H^2(\T)$ implies
  \begin{align*}
  \| H_0(k)^{1/2}(H(k)+2R)^{-1}\psi\|^2&=\langle (H(k)+2R)^{-1}\psi,H_0(k)(H(k)+2R)^{-1}\psi\rangle\\
  &=\langle (H(k)+2R)^{-1}\psi,\psi\rangle-\langle (H(k)+2R)^{-1}\psi, (V+2R)(H(k)+2R)^{-1}\psi\rangle\\
  &\leq \|(H(k)+2R)^{-1}\|\|\psi\|^2
  \end{align*}
  by the Cauchy-Schwarz inequality and since $\| V\|_{\infty}\leq R$.
  Since $\inf \sigma(H(k))\geq -R$, we have for $\|\varphi\|=1$,
  \begin{align*}
  	\|(H(k)+2R)^{-1}\varphi\|^2=\int_{-R}^\infty \frac{1}{(x+2R)^2}\dd\mu_{\varphi}(x)\leq \frac{1}{R^2}
   \end{align*}
   by the functional calculus. Thus, $\|(H(k)+2R)^{-1}\|\leq \frac1R$ and
  we may put these estimates together and conclude \eqref{eq:0414} with $M_1(R)=\frac{1}{R}+16$.
  \end{proof}

  \begin{remark}
    The extension of the bound \eqref{eq:1june22} from $H^2(\T)$ to $H^1(\T)$ would also follow from the theory of quadratic forms \cite{Kato,ReedSimon2}; indeed, relative boundedness of a self-adjoint operator yields relative boundedness of the associated closed quadratic form, which has a representation in terms of the positive square root of the operator.
  \end{remark}

Before proving Lemma~\ref{lem:1june29}, our second uniform estimate needed to prove Theorem~\ref{thm:Quantper}, we prove the auxiliary Lemmas~\ref{lem:pointwise} and \ref{lem:1june30}. The first is a simple pointwise estimate for functions in $H^1(\R)\cap \dom(|X|^s)$, $s>0$, which will allow us to minimize the assumptions on the regularity and decay of our initial states in many of the statements which follow. The second shows that the series defining $U\psi(k,q)$ for these initial states converges pointwise. We will use the Sobolev norms defined by
  \begin{align*}
  \|\psi\|_{H^n(\R)}^2=\|\hat\psi\|^2+\|k^n\hat{\psi}\|^2=
  \|\psi\|^2+\|\psi^{(n)}\|^2
  \end{align*}
  for $n=1,2$, where $\hat{\psi}$ denotes the Fourier transform of $\psi$, and $\psi^{(n)}$ denotes the $n'th$ weak derivative.

  \begin{lemma}\label{lem:pointwise}
    Let $\psi\in H^1(\R)\cap \dom(|X|^s)$ for $s>0$. Then, there is a constant $C$ with $|\psi(x)|\leq \frac{C}{1+|x|^{s/2}}$ for all $x\in \R$.
  \end{lemma}

  \begin{proof}
    By the definition of $\dom(|X|^s)$, we have for $\psi\in \dom(|X|^s)$, $\int_{\R}(1+|t|^{2s})|\psi(t)|^2\dd t<\infty$. Thus, there exists an $N>0$ such that for $|x|\geq N$
    \begin{align*}
      (1+|x|^{2s})\int_{|t|\geq |x|}|\psi(t)|^2\dd t \leq \int_{|t|\geq |x|}(1+|t|^{2s})|\psi(t)|^2\dd t \leq 1,
    \end{align*}
    and we may bound the tails of $\|\psi\|^2$ as
    \begin{align}\label{eq:1june21}
      \int_{|t|\geq |x|}|\psi(t)|^2\dd t \leq \frac{1}{1+|x|^{2s}}.
    \end{align}
    We now use that $\psi\in H^1(\R)$ to make this estimate a pointwise one. Indeed, by the fundamental theorem of calculus and the Cauchy-Schwarz inequality, we have for $x\geq N$,
    \begin{align*}
      |\psi(x)|^2&=-2\Re\left(\int_{x}^{\infty}\overline{\psi(t)}\psi'(t)\dd t\right)\leq 2\left(\int_{x}^\infty |\psi'(t)|^2\dd t \right)^{1/2}\left(\int_{x}^\infty |\psi(t)|^2\dd t \right)^{1/2}
    \end{align*}
    since $\psi\in H^1(\R)$ implies $\lim_{x\to\pm \infty}\psi(x)=0$. So, bounding the right hand side with \eqref{eq:1june21}, and using the inequality 
    \[
    \frac{1}{1+y^2}\leq \frac{2}{(1+y)^2}
    \]
    for $y\geq 0$,
    we have 
    \begin{align*}
    |\psi(x)|^2\leq \frac{2\sqrt{2}}{1+x^{s}}\| \psi\|_{H^1(\R)}
    \end{align*}
    for $x\geq N$.
   Thus, by a similar computation for $x\leq -N$ and taking a square root, we have for $|x|\geq N$,
    \begin{align*}
      |\psi(x)|\leq \frac{2^{5/4}\|\psi\|^{1/2}_{H^1(\R)}}{1+|x|^{s/2}}.
    \end{align*}
    The lemma then follows by taking $C=\max\{ 2^{5/4}\|\psi\|^{1/2}_{H^1(\R)}, (1+N^{s/2})\sup_{|x|<N}|\psi(x)| \}$.
  \end{proof}

  \begin{lemma}\label{lem:1june30}
    For $\psi\in L^2(\R)\cap L^1(\R)$, the series $\sum_{\ell\in\Z}e^{-ik(q+p\ell)}\psi(q+p\ell)$ defining $U\psi(k,q)$ is absolutely convergent for a.e. $q\in \T$ and all $k\in \T^*$.
    Furthermore, for $s>2$, $ \dom(|X|^s)\subset L^1(\R)$, and for $\psi\in H^1(\R)\cap \dom(|X|^s)$, this series is absolutely uniformly convergent for all $k\in \T^*$ and $q\in \T$. 
  \end{lemma}

  \begin{proof}

	Let $\psi\in L^2(\R)\cap L^1(\R)$. Then by the monotone convergence theorem and an affine change of variables, we have 
	    \begin{align}\label{eq:0916}
	    \begin{split}
      \int_{\T}\sum_{\ell\in\Z}|\psi(q+p\ell)|\dd q&=
      \sum_{\ell\in\Z}\int_{0}^p|\psi(q+p\ell)|\dd q
      =\sum_{\ell\in\Z}\int_{p\ell}^{p(\ell+1)}|\psi(q)|\dd q\\
      &=\int_{\R}|\psi(q)|\dd q<\infty
      \end{split}
    \end{align}
    so that $\sum_{\ell\in \Z}e^{-ik(q+p\ell)}\psi(q+p\ell)$ converges absolutely for all $k$ and a.e. $q$.

By the Cauchy-Schwarz inequality, for $\psi\in \dom(|X|^s)$ we have
    \begin{align*}
      \int_{\R}|\psi(q)|\dd q&=
      \int_{\R}|\psi(q)|(1+|q|^{s})(1+|q|^{s})^{-1}\dd q\\
      &\leq
      \left(\int_{\R}|\psi(q)|^2(1+|q|^{s})^2\dd q\right)^{1/2}
      \left(\int_{\R}\frac1{(1+|q|^{s})^{2}}\dd q\right)^{1/2}
    \end{align*}
    which is finite by our choice of $s$. Thus, $\dom(|X|^s)\subset L^1(\R)$.  
  For $\psi\in H^1(\R)\cap \dom(|X|^s)$, using Lemma~\ref{lem:pointwise}, we have
  \begin{align}\label{eq:july29}
    |\psi(q+p\ell)|\leq \frac{C}{1+|q+p\ell|^{s/2}}\leq CM_\ell
  \end{align}
  where
  \[ M_\ell=\begin{cases}\frac{1}{1+|p\ell|^{s/2}},&\ell\geq 0\\\frac{1}{1+|p(\ell+1)|^{s/2}},& \ell\leq -1 \end{cases}\]
  and absolute uniform convergence follows.

  \end{proof}

  We now prove our last lemma before the main result of this section. The purpose of this lemma is to give an $L^\infty(\T^*,\frac{\dd k}{|\T^*|}; L^2(\T,\dd q))$ estimate for quantities appearing in the proof of Theorem~\ref{thm:Quantper}, allowing the separation of an integral over $\T^*$ into one over a set on which we have a good time dependent estimate of the integrand, and another on a set whose measure decays in time by Lemma~\ref{lem:DOSCont}. 

  \begin{lemma}\label{lem:1june29}
    Let $R>0$, $s>2$, $p\geq \pi$ and suppose $\|V\|_{\infty}\leq R$. For $\psi\in \mathcal{D}_s$ and a set $Z\subset \T^*$ with $\leb(Z)=0$, we have
    \begin{align}\label{eq:1july5}
      \sup_{k\in \T^*\setminus Z}\|(H(k)+2R)U\psi(k,\cdot)\|_{L^2(\T,\dd q)}^2\leq (1+3R)^2M_2
    \end{align}
    where $M_2=M_2(\psi)$ is a constant.
  \end{lemma}
  \begin{proof}
    For $k\in \T^*$ and $\varphi\in H^2(\T)$, we find
    \begin{align}\label{eq:1june29}
    \begin{split}
      \|(H(k)+2R)\varphi\|_{L^2(\T,\dd q)}&\leq \|H_0(k)\varphi\|_{L^2(\T,\dd q)}+\|(V+2R)\varphi\|_{L^2(\T,\dd q)}\\
      &\leq (1+3R)(\|H_0(k)\varphi\|_{L^2(\T,\dd q)}+\|\varphi\|_{L^2(\T,\dd q)}).
      \end{split}
    \end{align}
   Thus, it will suffice to find a uniform bound on $\|H_0(k)U\psi(k,\cdot)\|$ and $\|U\psi(k,\cdot)\|$.

    Fix $s>2$ and $\psi\in \mathcal{D}_s$. There exists a set $Z\subset \T^*$ with $\leb(Z)=0$ and such that for $k\in \T^*\setminus Z$, 
    \begin{align*}
      H_0(k)U\psi(k,q)=-\sum_{\ell\in\Z}e^{-ik(q+p\ell)}\psi''(q+p\ell)
    \end{align*}
    in $L^2(\T,\dd q)$. By Lemma~\ref{lem:1june30}, the series on the right is absolutely convergent for $k\in \T^*\setminus Z$ and a.e. $q$.
   
    Using this expression, we may bound $|H_0(k)U\psi(k,q)|^2$ pointwise for a.e. $q$ as follows:
    \begin{align*}
      |H_0(k)U\psi(k,q)|^2
      &=\sum_{\ell\in\Z}\sum_{m\in \Z}\psi''(q+p\ell)\overline{\psi''(q+pm)}e^{-ikp(\ell-m)}\\
      &\leq\sum_{\ell\in\Z}\sum_{m\in \Z}|\psi''(q+p\ell)||\psi''(q+pm)|.
    \end{align*}
    Then, by the monotone convergence theorem,
    \begin{align*}
    \int_{\T}|H_0(k)U\psi(k,q)|^2\dd q&\leq \sum_{\ell\in\Z}\sum_{m\in \Z}\int_0^p|\psi''(q+p\ell)||\psi''(q+pm)|\dd q\\
    &\leq \sum_{\ell\in\Z}\left(\int_0^p|\psi''(q+p\ell)|^2\dd q\right)^{1/2}\sum_{m\in \Z}\left(\int_{0}^p|\psi''(q+pm)|^2\dd q\right)^{1/2}\\
    &=\sum_{\ell\in\Z}\left(\int_{p\ell}^{p(\ell+1)}|\psi''(q)|^2\dd q\right)^{1/2}\sum_{m\in \Z}\left(\int_{pm}^{p(m+1)}|\psi''(q)|^2\dd q\right)^{1/2}
    \end{align*}
    by the Cauchy-Schwarz inequality. For the $N$ of Lemma~\ref{lem:pointwise} corresponding to $\psi''\in \dom(|X|^s)$ and since $p\geq \pi$, we have for $\ell\geq N+1$
    \begin{align*}
      \int_{p\ell}^{p(\ell+1)}|\psi''(q)|^2\dd q\leq \frac{1}{1+p^{2s}|\ell|^{2s}}\leq \frac{1}{1+|\ell|^{2s}}
    \end{align*}
    by \eqref{eq:1june21}. Similarly, we have for $\ell\leq -(N+1)$
    \begin{align*}
    \int_{p\ell}^{p(\ell+1)}|\psi''(q)|^2\dd q\leq \frac{1}{1+|\ell+1|^{2s}}
    \end{align*}
    so that
    \begin{align*}
      \sum_{\ell\in\Z}\left(\int_{p\ell}^{p(\ell+1)}|\psi''(q)|^2\dd q\right)^{1/2}\leq
      \sum_{\ell=-N}^{N}\left(\int_{p\ell}^{p(\ell+1)}|\psi''(q)|^2\dd q\right)^{1/2}+\sum_{|\ell|\geq N}\frac{\sqrt{2}}{1+|\ell|^s}.
    \end{align*}

    Thus,
    \begin{align}
      \sum_{\ell\in\Z}\left(\int_{p\ell}^{p(\ell+1)}|\psi''(q)|^2\dd q\right)^{1/2}\leq
      (2N+1)\|\psi''\|_{L^2(\R)}+\sum_{|\ell|\geq N}\frac{\sqrt{2}}{1+|\ell|^s}
    \end{align}
    where we emphasize that $N$ only depends only on $\psi$. By the same argument, for the $N'$ corresponding to $\psi$,
    \begin{align}
      \|U\psi(k,\cdot)\|^2\leq (2N'+1)\|\psi\|_{L^2(\R)}+\sum_{|\ell|\geq N'}\frac{\sqrt{2}}{1+|\ell|^s}
    \end{align}
    and \eqref{eq:1july5} follows from \eqref{eq:1june29}.

  \end{proof}

  We are now ready to prove Theorem~\ref{thm:Quantper}. Its proof involves a comparison using the min-max theorem; we note here that the eigenvalues of $H_0(0)$
  are $\left\{\frac{(2\ell\pi)^2}{p^2} \right\}_{\ell=0}^\infty$, while those for $H_0(\frac{\pi}{p})$ are $\left\{\frac{((2\ell+1)\pi)^2}{p^2} \right\}_{\ell=0}^\infty$.

  \begin{proof}[Proof of Theorem~\ref{thm:Quantper}]
    By unitarity of $U$, it will suffice to bound the norm of
    \[U\left(\frac1t\int_0^tD(r)\psi\dd r\right)-UQ_{H}\psi\]
     in $L^2\left(\T^*, \frac{\dd k}{|\T^*|}; L^2(\T,\dd q)\right)$.

    We use the identity $P_n(k)(D+k)P_n(k)=\frac12\frac{\dd E_n(k)}{\dd k}P_n(k)$ for almost every $k\in \T^*$ \cite{AschKnauf} to compute
    \begin{align}\label{eq:0831212}
    \begin{split}
      &\left\|U\left(\frac1t\int_0^tD(r)\psi\dd r\right)-\left(\int_{\T^*}^\oplus\sum_{n=1}^\infty\frac12 \frac{\dd E_n(k)}{\dd k}P_n(k)\frac{\dd k}{|\T^*|}U\psi \right) \right\|^2\\
    &=\left\|U\left(\frac1t\int_0^tD(r)\psi\dd r\right)-\left(\int_{\T^*}^\oplus\sum_{n=1}^\infty P_n(k)(D+k)P_n(k)\frac{\dd k}{|\T^*|}U\psi \right) \right\|^2\\
    &=\left\|\left(\frac1t\int_0^tUD(r)U^*U\psi\dd r\right)-\left(\int_{\T^*}^\oplus\sum_{n=1}^\infty P_n(k)(D+k)P_n(k)\frac{\dd k}{|\T^*|}U\psi \right) \right\|^2.
    \end{split}
    \end{align}
    We now show 
    \begin{align}\label{eq:083121}
    UD(r)U^*U\psi=\int_{\T^*}^\oplus e^{irH(k)}(D+k)e^{-irH(k)}\frac{\dd k}{|\T^*|}U\psi
    \end{align} 
    as follows: Since
    \begin{align*}
    	Ue^{\pm irH}U^*=\int_{\T^*}^\oplus e^{\pm i rH(k)}\frac{\dd k}{|\T^*|}, \;\; UDU^*=\int_{\T^*}^\oplus (D+k)\frac{\dd k}{|\T^*|},
    \end{align*}
	 it will suffice to verify $\int_{\T^*}^\oplus e^{- i rH(k)}\frac{\dd k}{|\T^*|}U\psi$ is in the domain of 
    $\int_{\T^*}^\oplus (D+k)\frac{\dd k}{|\T^*|}$ for $\psi\in \mathcal{D}_s\subset H^2(\R)$. Since $\psi\in H^2(\R)$, for a full-measure set of $k\in \T^*$, $U\psi(k,\cdot)\in H^2(\T)$. Thus, for these $k$, $e^{-irH(k)}U\psi(k,\cdot)\in H^2(\T)$, and in particular, $\psi(k,\cdot)\in H^1(\T)$. Using that $\inf\sigma(H(k))\geq -R$ we may may compute $(H(k)+2R)^{-1/2}$ and find 
    \begin{align*}
    \int_{\T^*}&\left\|(D+k)e^{-irH(k)}U\psi(k,\cdot) \right\|^2\frac{\dd k}{|\T^*|}\\
    &= \int_{\T^*}\left\|(D+k)e^{-irH(k)}(H(k)+2R)^{-1/2}(H(k)+2R)^{1/2}U\psi(k,\cdot) \right\|^2\frac{\dd k}{|\T^*|}\\
    &\leq 
        \int_{\T^*}\left\|(D+k)(H(k)+2R)^{-1/2}\right\|^2\left\|e^{-irH(k)}U((H+2R)^{1/2}\psi)(k,\cdot) \right\|^2\frac{\dd k}{|\T^*|}
    \end{align*}
since $(H(k)+2R)^{1/2}U\psi=U(H+2R)^{1/2}\psi$ for a.e. $k\in \T^*$. Finally, using the bound 
\[\| (D+k)(H(k)+2R)^{-1/2}\|\leq 1,\]
 and that $U$ is an isometry, we may estimate
 \begin{align*}
 \int_{\T^*}&\left\|(D+k)e^{-irH(k)}U\psi(k,\cdot) \right\|^2\frac{\dd k}{|\T^*|}\leq 
\int_{\T^*}\left\|U((H+2R)^{1/2}\psi)(k,\cdot) \right\|^2\frac{\dd k}{|\T^*|}\\
&=\|(H+2R)^{1/2}\psi\|^2_{L^2(\R)}\leq (1+3R)\|\psi\|^2_{H^1(\R)}.
 \end{align*}

   Thus, using \eqref{eq:083121}, we may expand the final expression in \eqref{eq:0831212} in terms of eigenprojections twice to find
    \begin{align*}
    \begin{split}
   \left\|\left(\frac1t\int_0^t\int_{\T^*}^\oplus\sum_{m=1}^\infty e^{iE_m(k)r}P_m(k)(D+k)\sum_{n=1}^\infty e^{-iE_n(k)r}P_n(k)\frac{\dd k}{|\T^*|}U\psi\dd r\right)\right.\\
  \left.-\left(\int_{\T^*}^\oplus\sum_{n=1}^\infty P_n(k)(D+k)P_n(k)\frac{\dd k}{|\T^*|}U\psi \right) \right\|^2.   
   \end{split}
    \end{align*}
    Using that $\psi\in H^2(\R)$ again, we may justify interchanging the sum and $(D+k)$ for each $k\not\in \frac{\pi}{p}\Z$ with $U\psi(k,\cdot)\in H^2(\T)$, as follows:
    \begin{align*}
    (D+k)\sum_{n=1}^\infty e^{-iE_n(k)s}P_n(k)U\psi&=(D+k)(H(k)+2R)^{-1/2}\sum_{n=1}^\infty e^{-iE_n(k)s}
    P_n(k)(H(k)+2R)^{1/2}U\psi\\
    &=\sum_{n=1}^\infty e^{-iE_n(k)s}(D+k)
    P_n(k)U\psi,
    \end{align*}
    since 
    $\|(D+k)(H(k)+2R)^{-1/2}\|\leq 1/\sqrt{R}$,
	as in the proof of Lemma~\ref{lem:relbound}. Thus, passing the projections $P_m(k)$ through the sum in $n$, we find
    \begin{align*}
 \begin{split}
  & \left\|\left(\frac1t\int_0^t\int_{\T^*}^\oplus\sum_{m=1}^\infty e^{iE_m(k)r}P_m(k)(D+k)\sum_{n=1}^\infty e^{-iE_n(k)r}P_n(k)\frac{\dd k}{|\T^*|}U\psi\dd r\right)\right.\\
 & \qquad  \qquad \qquad\qquad \left.-\left(\int_{\T^*}^\oplus\sum_{n=1}^\infty P_n(k)(D+k)P_n(k)\frac{\dd k}{|\T^*|}U\psi \right) \right\|^2  \\
       &=\int_{\T^*}\left\|\left(\frac1t\int_0^t\sum_{m=1}^\infty\sum_{n=1}^\infty e^{i(E_m(k)-iE_n(k))r}P_m(k)(D+k)P_n(k)U\psi\dd r\right)\right.\\
      &\qquad  \qquad \qquad\qquad-\left.\left(\sum_{n=1}^\infty P_n(k)(D+k)P_n(k)U\psi \right) \right\|^{2}\frac{\dd k}{|\T^*|}\\
   & = \int_{\T^*}\left\|\frac1t\int_{0}^t\sum_{m=1}^\infty \sum_{\substack{n=1\\n\ne m}}^\infty P_m(k)e^{i(E_m(k)-E_n(k))r}(D+k)P_n(k)U\psi(k,\cdot)\dd r \right\|^2\frac{\dd k}{|\T^*|}   		 \end{split}	
    \end{align*}
cancelling the diagonal terms. We commute the time integral with the sum in $m$ using dominated convergence for the Bochner integral \cite{HillePhillips}. Namely, we note that for $M\in \N$ and a.e. $k$, we have by orthogonality of the projections 
\begin{align*}
&\left\|\sum_{m=1}^{M}\sum_{\substack{n=1\\n\ne m}}^\infty P_m(k)e^{i(E_m(k)-E_n(k))r}(D+k)P_n(k)U\psi(k,\cdot)\right\|^2\\
&\leq \sum_{m=1}^\infty \left\|P_m(k)\sum_{\substack{n=1\\n\ne m}}^\infty e^{-iE_n(k)r}(D+k)P_n(k)U\psi(k,\cdot)\right\|^2\\
    &\leq \|(D+k)(H(k)+2R)^{-1/2} \|^2\sum_{\substack{n=1\\n\ne m}}^\infty \left\| P_n(k)(H(k)+2R)^{1/2}U\psi(k,\cdot)\right\|^2\\
    &\leq  \|(D+k)(H(k)+2R)^{-1/2} \|^2 \left\| (H(k)+2R)^{1/2}U\psi(k,\cdot)\right\|^2\\
\end{align*}
so that 
\begin{align*}
\sup_{r\in[0,t]}&\left\|\sum_{m=1}^{M}\sum_{\substack{n=1\\n\ne m}}^\infty P_m(k)e^{i(E_m(k)-E_n(k))r}(D+k)P_n(k)U\psi(k,\cdot)\right\|\\
&\leq \|(D+k)(H(k)+2R)^{-1/2} \| \left\| (H(k)+2R)^{1/2}U\psi(k,\cdot)\right\|
\end{align*}
which is finite for a.e. $k$. Thus, using uniform convergence to pass the integral through the sum in $n$, we have 
\begin{align*}
 \int_{\T^*}&\left\|\frac1t\int_{0}^t\sum_{m=1}^\infty \sum_{\substack{n=1\\n\ne m}}^\infty P_m(k)e^{i(E_m(k)-E_n(k))r}(D+k)P_n(k)U\psi(k,\cdot)\dd r \right\|^2\frac{\dd k}{|\T^*|}   \\
 &= \int_{\T^*}\sum_{m=1}^\infty\left\| P_m(k)\sum_{\substack{n=1\\n\ne m}}^\infty\frac1t\int_0^t e^{i(E_m(k)-E_n(k))r}\dd r(D+k)P_n(k)U\psi(k,\cdot)\right\|^2\frac{\dd k}{|\T^*|} 
\end{align*}
by the orthogonality of the projections $P_m(k)$.

  We begin by bounding the integrand above pointwise for all $m$ and a.e. $k$. For $k\notin\frac{\pi}{p}\Z$, let $v_\ell(k)$ be the basis of normalized eigenfunctions of $H(k)$, i.e. for $\varphi\in L^2(\T,\dd q)$, $P_{\ell}(k)\varphi=\langle \varphi,v_\ell(k)\rangle_{L^2(\T,\dd q)}v_{\ell}(k)$.
  Then, defining the shorthand 
  \begin{align*}
 f_{m,n}=f_{m,n}(t,k):=\frac1t\int_0^te^{i(E_m(k)-E_n(k))r}\dd r
  \end{align*}
  in the above, we have
  \begin{align*}
  & \int_{\T^*}\sum_{m=1}^\infty\left\| P_m(k)\sum_{\substack{n=1\\n\ne m}}^\infty f_{m,n}(D+k)P_n(k)U\psi(k,\cdot)\right\|^2\frac{\dd k}{|\T^*|}\\
  &= \int_{\T^*}\sum_{m=1}^\infty\left|\left\langle \sum_{\substack{n=1\\n\ne m}}^{\infty}f_{m,n}(D+k)P_n(k)U\psi(k,\cdot),v_m(k) \right\rangle\right|^2\frac{\dd  k}{|\T^*|}\\
  &= \int_{\T^*}\sum_{m=1}^\infty\left|\left\langle (H(k)+2R)^{-1/2}\sum_{\substack{n=1\\n\ne m}}^{\infty}f_{m,n}(H(k)+2R)^{1/2}(D+k)P_n(k)U\psi(k,\cdot),v_m(k) \right\rangle\right|^2\frac{\dd k}{|\T^*|}\\
  &= \int_{\T^*}\sum_{m=1}^\infty\left|\left\langle\sum_{\substack{n=1\\n\ne m}}^{\infty}f_{m,n}(H(k)+2R)^{1/2}(D+k)P_n(k)U\psi(k,\cdot),\frac{1}{\sqrt{E_m(k)+2R}}v_m(k) \right\rangle\right|^2\frac{\dd k}{|\T^*|},
\end{align*}  
  using that $(H(k)+2R)^{-1/2}$ is self-adjoint, and each $v_\ell(k)$ is an eigenfunction in the last line. Thus, 
  \begin{align*}
   & \int_{\T^*}\sum_{m=1}^\infty\left\| P_m(k)\sum_{\substack{n=1\\n\ne m}}^{\infty}f_{m,n}(D+k)P_n(k)U\psi(k,\cdot)\right\|^2\frac{\dd k}{|\T^*|}\\
&  \leq  \int_{\T^*}\sum_{m=1}^\infty\left\| \sum_{\substack{n=1\\n\ne m}}^{\infty}f_{m,n}(H(k)+2R)^{1/2}(D+k)P_n(k)U\psi(k,\cdot)\right\|^2\frac{1}{E_m(k)+2R}\frac{\dd k}{|\T^*|}
  \end{align*}
  by the Cauchy-Schwarz inequality.  Then, using \eqref{eq:0414}, we may insert a factor of $(H(k)+2R)^{-1}(H(k)+2R)=I$ in the above to find 
  \begin{align*}
    \int_{\T^*}&\sum_{m=1}^\infty \left\| \sum_{\substack{n=1\\n\ne m}}^{\infty}f_{m,n}(H(k)+2R)^{1/2}(D+k)P_n(k)U\psi(k,\cdot)\right\|^2\frac{\dd k}{|\T^*|}\\
   & \leq M_1 \int_{\T^*}\sum_{m=1}^\infty\left\| \sum_{\substack{n=1\\n\ne m}}^{\infty}f_{m,n}P_n(k)(H(k)+2R)U\psi(k,\cdot)\right\|^2\frac{\dd k}{|\T^*|}\\
    &=M_1 \int_{\T^*}\sum_{m=1}^\infty\sum_{\substack{n=1\\n\ne m}}^{\infty}\left\| f_{m,n}P_n(k)(H(k)+2R)U\psi(k,\cdot)\right\|^2\frac{\dd k}{|\T^*|}
  \end{align*}
  by orthogonality of the projections $P_n(k)$.
  In summary, we thus far have shown
  \begin{align}\label{eq:1july19}
    \begin{split}
   & \left\|\left(\frac1t\int_0^tD(r)\psi\dd r\right)-Q_H\psi \right\|^2\\
    &\leq M_1\int_{\T^*}\sum_{m=1}^\infty\sum_{\substack{n=1\\n\ne m}}^{\infty}\biggl\| f_{m,n} P_n(k)(H(k)+2R)U\psi(k,\cdot)\biggr\|^2 \frac{1}{E_m(k)+2R}\frac{\dd k}{|\T^*|}.
    \end{split}
  \end{align}

  Before integrating the right hand side of \eqref{eq:1july19} in $k$, we partition the Brillouin zone into the sets of ``good" and ``bad" $k\in \mathcal{B}$. Using the notation of Lemma~\ref{lem:DOSCont}, we define
  \[
  B_t:=\mathcal{B}_{\epsilon_t},\quad \epsilon_t:= \frac1{C_2\pi}t^{-4/5}.
  \]
 Then, for $k\in G_t:=\mathcal{B}\setminus B_t$, computing the integral defining $f_{m,n}$, we have
  \begin{align}
  |f_{m,n}|\leq \frac{2}{t|E_m(k)-E_n(k)|}\leq \frac{2}{t\epsilon_t}
  \end{align}
  so that 
  \begin{align*}
  & \int_{G_t}\sum_{m=1}^\infty\sum_{\substack{n=1\\n\ne m}}^{\infty}\left\| f_{m,n}P_n(k)(H(k)+2R)U\psi(k,\cdot)\right\|^2\frac{1}{E_m(k)+2R}\frac{\dd k}{|\T^*|}\\
  &\leq  \frac{4}{t^2}\int_{G_t}\sum_{m=1}^\infty\sum_{\substack{n=1\\n\ne m}}^{\infty}\frac{\|P_n(k)(H(k)+2R)U\psi(k,\cdot)\|^2}{|E_m(k)-E_n(k)|^2}\frac{1}{E_m(k)+2R}\frac{\dd k}{|\T^*|}\\
  &\leq \frac{4}{t^2\epsilon_t^2}\int_{G_t}\sum_{m=1}^\infty\frac{\|(H(k)+2R)U\psi(k,\cdot)\|^2_{L^2(\T,\dd q)}}{E_m(k)+2R}\frac{\dd k}{|\T^*|}
  \end{align*}
  by Bessel's inequality.

  Instead, for $k\in B_t$, now using 
 $ \left|\frac1t\int_0^t e^{i(E_m(k)-E_n(k))r}\dd r\right|\leq 1$,
  we have
  \begin{align*}
  &\int_{B_t}\sum_{m=1}^\infty\sum_{\substack{n=1\\n\ne m}}^{\infty}\left\| f_{m,n}P_n(k)(H(k)+2R)U\psi(k,\cdot)\right\|^2\frac{1}{E_m(k)+2R}\frac{\dd k}{|\T^*|}\\
  &\leq \int_{B_t}\sum_{m=1}^\infty\frac{\|(H(k)+2R)U\psi(k,\cdot)\|^2_{L^2(\T,\dd q)}}{E_m(k)+2R}\frac{\dd k}{|\T^*|}.
  \end{align*}

  Summarizing once again, we have
  \begin{align*}
    \left\|\left(\frac1t\int_0^tD(r)\psi\dd r\right)-Q_{H}\psi\right\|^2
    &\leq
    M_1\int_{B_t}\sum_{m=1}^\infty \frac{\|(H(k)+2R)U\psi(k,\cdot)\|^2}{E_m(k)+2R}\frac{\dd k}{|\T^*|}\\
    &+\frac{4M_1}{t^2\epsilon_t^2}\int_{G_t}\sum_{m=1}^\infty \frac{\|(H(k)+2R)U\psi(k,\cdot)\|^2}{E_m(k)+2R}\frac{\dd k}{|\T^*|},
  \end{align*}
  which after an application of Lemma~\ref{lem:1june29} becomes
  \begin{align}\label{eq:1june23}
    \begin{split}
    \left\|\left(\frac1t\int_0^tD(r)\psi\dd r\right)-Q_{H}\psi\right\|^2
    &\leq (1+3R)^2M_1M_2\int_{B_t}\sum_{m=1}^\infty \frac{1}{E_m(k)+2R}\frac{\dd k}{|\T^*|}\\
    &+\frac{4(1+3R)^2M_1M_2}{t^2\epsilon_t^2}\int_{G_t}\sum_{m=1}^\infty \frac{1}{E_m(k)+2R}\frac{\dd k}{|\T^*|}.
    \end{split}
  \end{align}

  Before summing in $m$ and integrating in $k$, we find an upper bound on $1/(E_m(k)+2R)$ which is uniform in $k$ and summable in $m$. Since $(-1)^{n+1}\frac{\dd E_n}{\dd k}>0$ on $\left(0,\frac{\pi}{p}\right)$ and $E_n(k)=E_n(-k)$, we have
  for $n\in\N$,
  \begin{align*}
  E_{2n-1}(k)+2R&\geq E_{2n-1}(0)+2R
  \geq \frac{4(n-1)^2\pi^2}{p^2}+R\geq \frac{\pi^2}{p^2}(4(n-1)^2+R)
  \end{align*}
  by the min-max theorem and since
  $p\geq \pi$.
  Similarly,
  \begin{align*}
  E_{2n}(k)+2R\geq E_{2n}(\pi/p)+2R\geq \frac{(2n-1)^2\pi^2}{p^2}+R\geq \frac{\pi^2}{p^2}((2n-1)^2+R),
  \end{align*}
  so that
  \begin{align*}
    \sum_{m=1}^\infty\frac{1}{E_m(k)+2R}\leq \frac{p^2}{\pi^2}A,\quad \text{for}\; A=A(R):=\left(\sum_{n=1}^\infty\frac{1}{4(n-1)^2+R}+\frac{1}{(2n-1)^2+R}\right).
  \end{align*}
  Finally, we note that by Lemma~\ref{lem:DOSCont},
  \begin{align*}
    \leb(B_t)\leq 4\sqrt{C_2\pi}\exp(C_2R^{1/2}p)\sqrt{\epsilon_t}=
   4\exp(C_2R^{1/2}p)t^{-2/5}.
  \end{align*}
  With these facts in hand, and defining $M=M(R):=4(1+3R)^2M_1A$, we may bound the integrals in \eqref{eq:1june23} as follows:
  \begin{align*}
  \left\|\left(\frac1t\int_0^tD(r)\psi\dd r\right)-Q_{H}\psi\right\|^2
  &\leq \frac{p^2}{\pi^2} M\cdot M_2\left( \frac{p}{2\pi}\leb(B_t)+\frac{1}{\epsilon_t^2t^2}\right)\\
  &\leq \frac{p^2}{\pi^2}M\cdot M_2\left( \frac{2p}{\pi}\exp\left(C_2R^{1/2}p\right)+C_2^2\pi^2\right)t^{-2/5}\\
  &\leq \frac{M\cdot M_2}{\pi^2}\left( \frac2{\pi}+C_2^2\pi^2\right)p^3\exp(C_2R^{1/2}p)t^{-2/5}
  \end{align*}
  since $\frac{p}{2\pi}\geq \frac14$.
  The result then follows from taking a square root and setting $M_3=M_2^{1/2}$ and $C_1=\frac{M^{1/2}}{\pi}\left( \frac{2}{\pi}+C_2^2\pi^2\right)^{1/2}$.
  \end{proof}

  \section{Proof of Ballistic transport}

  In this section, we will prove Theorem~\ref{thm:main}. First, we must prove some final lemmas. The first is an elementary propagation estimate, which we record below. 

  \begin{lemma}\label{lem:generalprop}
  For $H_i=-\frac{d^2}{dx^2}+V_i$, and $V_i\in L^\infty(\R)$,
  \begin{align}
  \|e^{itH_1}-e^{itH_2}\|\leq |t|\|V_1-V_2\|_\infty.
  \end{align}
  \end{lemma}
  \begin{proof}
  Let $\varphi\in H^2(\R)$. Using strong differentiability of $(I-e^{-itH_1}e^{itH_2})\varphi$, we have by the fundamental theorem of calculus and unitarity of $e^{itH_i}$,
  \begin{align*}
  \|(e^{itH_1}-e^{itH_2})\varphi\|&=\left\|\int_0^te^{-irH_1}(V_1-V_2)e^{irH_2}\varphi\dd r\right\|\\
  &\leq \left|\int_0^t\|V_1-V_2\|_\infty\|\varphi\|\dd r \right|\\
  &= |t|\|V_1-V_2\|_\infty\|\varphi\|.
  \end{align*}
  Since $H^2(\R)$ is dense in $L^2(\R)$, the result follows by approximation.
  \end{proof}

  We will also require the following estimate showing that the difference of the Heisenberg evolution of the position operator for two Schr\"odinger operators with bounded potentials applied to a suitable state can diverge at most quadratically in time. As we will soon see, the presence of the normed difference of the two potentials on the right hand side of the estimate is critical to the proof of Theorem~ \ref{thm:main}.

  \begin{lemma}\label{lem:posprop}
  Let $H_i$ be as in Lemma~\ref{lem:generalprop}, and suppose $V_1-V_2\in L^\infty(\R)$. Then, for $\psi\in H^2(\R)\cap D(X^2)$ and $t\in\R$,
  \[
  \|X_{H_1}(t)\psi-X_{H_2}(t)\psi\|\leq \Gamma(|t|^{1/2}+t^2)\|V_1-V_2\|_{\infty}^{1/2}
  \]
  for $\Gamma=\Gamma(R,\psi)$ a constant.
  \end{lemma}
  \begin{proof}
  Let
  \[
  \eta_\psi(t):=X_{H_1}(t)\psi-X_{H_2}(t)\psi
  \]
  for $\psi\in H^2(\R)\cap D(X^2)$.

  We define $F_\epsilon(x)=\frac{x}{1+\epsilon|x|}$ and 
  \[
  \eta_{\psi}^\epsilon:=e^{itH_1}F_\epsilon(X)e^{-itH_1}\psi-e^{itH_2}F_\epsilon(X)e^{-itH_2}\psi,\] 
  which we will estimate above for each $\epsilon>0$ and take $\epsilon\searrow 0$.

  We note that for $i=1,2$,
  \begin{align*}
  \|e^{itH_i}F_\epsilon(X)e^{-itH_i}\psi-e^{itH_i}Xe^{-itH_i}\psi\|^2&=\int_{\R}(F_\epsilon(x)-x)^2|(e^{-itH_i}\psi)(x)|^2\dd x\\
  &=\int_\R\frac{\epsilon^2x^4}{(1+\epsilon|x|)^2}|(e^{-itH_i}\psi)(x)|^2\dd x.
  \end{align*}

  Since for $\epsilon\leq 1$, the integrand is bounded above by $x^4|(e^{-itH_i}\psi)(x)|^2$, which is integrable for each $t\in \R$ by \cite{RadinSimon}[Theorem 2.2], we have
  \begin{align}
  \lim_{\epsilon\searrow 0}\|e^{itH_i}F_\epsilon(X)e^{-itH_i}\psi-e^{itH_i}Xe^{-itH_i}\psi\|=0
  \end{align}
  by the dominated convergence theorem. In particular, 
 $ \lim_{\epsilon\searrow 0}\eta_{\psi}^\epsilon(t)=\eta_{\psi}(t).$

  We find by the triangle inequality and Lemma~\ref{lem:generalprop},
  \begin{align*}
  \| \eta_\psi^\epsilon(t)\|\leq |t|\|V_1-V_2\|_\infty\|F_\epsilon(X)e^{-itH_1}\psi\|+
  \|F_\epsilon(X)(e^{-itH_1}-e^{-itH_2})\psi\|.
  \end{align*}
  Using the Cauchy-Schwarz inequality and that $F_\epsilon(X)$ is a bounded self-adjoint operator, we have
  \begin{align*}
  \|F_\epsilon(X)(e^{-itH_1}-e^{-itH_2})\psi\|^2&\leq
  \|e^{-itH_1}-e^{-itH_2}\| \|\psi\| \cdot \|F_\epsilon(X)^2(e^{-itH_1}-e^{-itH_2})\psi\|\\
  &\leq |t|\|V_1-V_2\|_\infty\|\psi\|\left(\| F_\epsilon(X)^2e^{-itH_1}\psi\|+\|F_\epsilon(X)^2e^{-itH_2}\psi\|\right)
  \end{align*}
  by Lemma~\ref{lem:generalprop}. Thus, taking $\epsilon\searrow 0$ and using the monotone convergence theorem for the terms on the right hand side, we have
  \begin{align*}
  \|\eta_{\psi}(t)\|&\leq |t|\|V_1-V_2\|_\infty\|Xe^{-itH_1}\psi\|\\
  &+|t|^{1/2}\|V_1-V_2\|_{\infty}^{1/2}\|\psi\|^{1/2}
  \left(\|X^2e^{-itH_1}\psi\|+\|X^2e^{-itH_2}\psi\|\right)^{1/2}  \\
  &\leq \Gamma(|t|^{1/2}+t^2)\|V_1-V_2\|_{\infty}^{1/2}
  \end{align*}
  where $\Gamma=\Gamma(R,\psi)$ is a constant by Theorems \ref{thm:RadinSimon1} and \ref{thm:RadinSimon2}.
  \end{proof}

	Before proving Theorem~\ref{thm:main}, we prove a final lemma, which provides a lower bound on the Floquet transform $U_{p_n}\psi$ integrated over a subset of $\mathcal{B}_n:=\left(-\frac{\pi}{p_n},\frac{\pi}{p_n} \right]$ for large enough periods $p_n$. To emphasize the dependence on $n$, we also denote $\T_n:=\R/p_n\Z$ and $\T^*_n:=\R/\frac{2\pi}{p_n}\Z$. 

  \begin{lemma}\label{lem:fourierbound}
  Let $\psi\in H^1(\R)\cap \dom(|X|^s)$ for $s>2$ and define the set $S_n:=\left[-\frac{3\pi}{4p_n},-\frac{\pi}{4p_n} \right]\cup \left[\frac{\pi}{4p_n},\frac{3\pi}{4p_n} \right]$. Then, there exists an $N_0(\psi)$ so that
  for $n\geq N_0(\psi)$,
  \begin{align}\label{eq:fourierbound}
  \int_{S_n}\|U_{p_n}\psi(k,\cdot)\|^2\frac{\dd k}{|\T_n^*|}\geq \frac{1}{8}\|\psi\|^2.
  \end{align}
  \end{lemma}
  \begin{proof}
	By Lemma~\ref{lem:1june30}, there is a set $Z_n\subset \T_n^*$ with $\leb(Z_n)=0$ and such that for a given $k\in \T_n^*\setminus Z_n$, 
    \begin{align*}
    U_{p_n}\psi(k,q)=e^{-ikq}(\psi(q)+e^{ikp_n}\psi(q-p_n))+\sum_{\ell\not\in\{0,-1\}}e^{-ik(q+p_n\ell)}\psi(q+p_n\ell)
    \end{align*}
    for a.e. $q$.
    We will use the first two terms in the above to estimate the integral in \eqref{eq:fourierbound} from below. For these $k$ and $q$,
    \begin{align}
    |U_{p_n}\psi(k,q)-e^{-ikq}(\psi(q)+e^{ikp_n}\psi(q-p_n))|\leq \sum_{\ell\not\in\{0,-1\}}|\psi(q+p_n\ell)|,
    \end{align}
    and by \eqref{eq:july29}
    \begin{align*}
    |\psi(q+p_n\ell)|\leq \frac{C}{p_n}\left(\frac{1}{|\ell|^{s/2}}+\frac{1}{|\ell+1|^{s/2}} \right)
    \end{align*}
    for
    $\ell\not\in\{0,-1\}$. Thus, since $p_n\to \infty$ as $n\to \infty$,
    \begin{align}\label{eq:1june13}
    \lim_{n\to \infty}\sup_{k\in \T^*_n\setminus Z_n}\|U_{p_n}\psi(k,q)-e^{-ikq}(\psi(q)+e^{ikp_n}\psi(q-p_n))\|^2_{L^2(\T_n,\dd q)}=0.
    \end{align}

    By Fubini's theorem and the change of variables $p_nk\mapsto k$, for $S:=p_nS_n$,
    \begin{align*}
      &\int_{S_n}\|\psi(q)+e^{ikp_n}\psi(q-p_n)\|^2\frac{\dd k}{|\T^*_n|}\\
      &=
      \int_{\T_n}\int_{S}|\psi(q)+e^{ik}\psi(q-p_n)|^2\frac{\dd k}{2\pi}\dd q\\
      &=
      \int_{\T_n}\int_{S}|\psi(q)|^2+|\psi(q-p_n)|^2+2\Re(\overline{\psi}(q)e^{ik}\psi(q-p_n))\frac{\dd k}{2\pi}\dd q\\
      &=\frac12\int_{\T_n}|\psi(q)|^2+|\psi(q-p_n)|^2\dd q,
    \end{align*}
    since
    \begin{align*}
    \int_S\Re(\overline{\psi(q)}e^{ik}\psi(q-p_n))\dd k=
    \Re\left(\overline{\psi(q)}\psi(q-p_n)\int_Se^{ik}\dd k\right)=0.
    \end{align*}
    Thus,
    \begin{align}\label{eq:2june13}
    \int_{S_n}\|\psi(q)+e^{ikp_n}\psi(q-p_n)\|^2\frac{\dd k}{|\T^*_n|}= \frac{1}{2}\int_{-p_n}^{p_n}|\psi(q)|^2\dd q.
    \end{align}

    Since $p_n\to \infty$, $\lim_{n\to\infty}\int_{-p_n}^{p_n}|\psi(q)|^2\dd q=\|\psi\|^2$, and we may use the triangle inequality to combine \eqref{eq:1june13} and \eqref{eq:2june13} to find \eqref{eq:fourierbound}.

  \end{proof}

  We are now ready to prove Theorem~\ref{thm:main}. The proof will largely follow the path laid out by Fillman in  \cite{Fillman}. Indeed, to extract the limit $Q\psi$ in Theorem~\ref{thm:main}, we define
  \begin{align}\label{eq:july20}
  Q_n\psi:=\lim_{t\to\infty}\frac1tX_{H_n}(t)\psi,
  \end{align}
  for $H_n=-\frac{\dd^2}{\dd x^2}+V_n$ where $V_n$ approximate $V$ exponentially quickly, and show $Q_n\psi$ has a limit, denoted $Q\psi$. We then show that $Q\psi=\lim_{t\to\infty}X_{H}(t)\psi$. Finally, to conclude $Q\psi\ne 0$, we show $Q_n\psi\to Q\psi$ faster than $\|Q_n\psi\|$ can tend to $0$.

  \begin{proof}[Proof of Theorem~\ref{thm:main}]

  Fix $s>2$ and a $V$ such that $\|V\|_\infty\leq R$ and $V\in EC(\eta)$ with $\eta> 12\kappa C_2R^{1/2}$ for a constant $\kappa\geq 9/2$. Let $V_n$ with $V_n(\cdot +p_n)=V_n(\cdot)$ be such that
  \begin{align}\label{eq:expapprox}
  \lim_{n\to \infty}e^{\eta p_{n+1}}\|V_n-V\|_\infty=0.
  \end{align}
  By enlarging $R$ if necessary, we may assume $\|V_n\|_\infty\leq R$ for each $n\in \N$.

  Let $\psi\in \mathcal{D}_s$. Defining $H_n=-\frac{\dd^2}{\dd x^2}+V_n$ with domain $H^2(\R)$, we denote $Q_n\psi$ as in \eqref{eq:july20}
  where the limit exists for $\psi\in H^1(\R)\cap \dom(X)\subset \mathcal{D}_s$ by \cite{AschKnauf} (or Theorem~\ref{thm:Quantper}).

  We find a limit for the $Q_n\psi$ using a telescoping sum; namely we show that
  \begin{align}\label{eq:summable}
  \sum_{n=1}^\infty\|Q_{n+1}\psi-Q_n\psi\|<\infty.
  \end{align}
  To this end, we set
  \[
  t_n=C_1^{5}p_{n+1}^{15/2}\exp\left(5\kappa C_2R^{1/2}p_{n+1}\right).
  \]
  Since $p_n\to \infty$, we may take $p_n\geq \pi$ throughout what follows.
  Then, denoting by $D_n(t)=e^{itH_n}De^{-itH_n}$, we have by Theorem~\ref{thm:Quantper},
  \begin{align}\label{eq:1june10}
  \begin{split}
  \left\|\frac12Q_{n+1}\psi-\frac{1}{t_{n}}\int_0^{t_{n}}D_{n+1}(r)\psi\dd r \right\|&\leq C_1M_3p_{n+1}^{3/2}\exp\left(\frac12C_2R^{1/2}p_{n+1}\right)t_{n}^{-1/5}\\
  &=M_3\exp\left(\frac12\left(1-2\kappa\right)C_2R^{1/2}p_{n+1} \right).
  \end{split}
  \end{align}
  Similarly,
  \begin{align}\label{eq:2june10}
  \begin{split}
  \left\|\frac12Q_n\psi-\frac{1}{t_n}\int_0^{t_n}D_n(r)\psi\dd r \right\|&\leq
  C_1M_3p_n^{3/2}\exp\left(\frac12C_2R^{1/2}p_n\right)t_n^{-1/5}\\
  &\leq M_3\exp\left(\frac12\left(1-2\kappa\right)C_2R^{1/2}p_{n+1} \right),
  \end{split}
  \end{align}
  using that $p_{n+1}>p_n$ for all $n\in \N$.
 We find an $N_0(R)$ large enough to ensure $t_n\geq 1$ for $n\geq N_0(R)$. Then, by \eqref{eq:integral} and Lemma~\ref{lem:posprop},
  \begin{align}\label{eq:3june10}
  \begin{split}
  \frac{1}{t_n}\left\| \int_0^{t_n}D_n(r)\psi-D_{n+1}(r)\psi\dd r \right\|&\leq \Gamma(1/\sqrt{t_n}+t_n)\| V_n-V_{n+1}\|_{\infty}^{1/2}\\
  &\leq 2\Gamma t_n\| V_n-V_{n+1}\|_{\infty}^{1/2}
  \end{split}
  \end{align}
  for $n\geq N_0(R)$.

  Summing \eqref{eq:1june10}, \eqref{eq:2june10}, we find,
  \begin{align*}
  &\left\|\frac12Q_n\psi-\frac{1}{t_n}\int_0^{t_n}D_n(r)\psi\dd r \right\|+
  \left\|\frac12Q_{n+1}\psi-\frac{1}{t_{n}}\int_0^{t_{n}}D_{n+1}(r)\psi\dd r \right\|\\
  &\leq 2M_3\exp\left(\frac12\left(1-2\kappa\right)C_2R^{1/2}p_{n+1} \right)\leq
  2M_3\exp\left(-4C_2R^{1/2}p_{n+1} \right)
  \end{align*}
  where the last inequality uses $\kappa\geq 9/2$.

  By \eqref{eq:expapprox}, for $n$ large, we have
  $\|V_n-V_{n+1}\|_\infty^{1/2}\leq e^{-\frac{\eta}{2}p_{n+1}}$.
  So, by enlarging $N_0(R)$ if necessary, \eqref{eq:3june10} yields for $n\geq N_0(R)$,
  \begin{align*}
  \frac{1}{t_n}\left\| \int_0^{t_n}D_n(r)\psi-D_{n+1}(r)\psi\dd r \right\|&\leq 2\Gamma t_ne^{-\frac{\eta}{2}p_{n+1}}\leq
  2\Gamma t_n\exp\left(-6\kappa C_2R^{1/2}p_{n+1}\right)\\
  &= 2\Gamma C_1^5p_{n+1}^{15/2}\exp\left(-\kappa C_2R^{1/2}p_{n+1}\right)
  \end{align*}
  since $\eta> 12\kappa C_2R^{1/2}$. Increasing $N_0(R)$ again, we may take
  \begin{align*}
  2C_1^5p_{n+1}^{15/2}\exp\left(-\kappa C_2R^{1/2}p_{n+1}\right)\leq \exp\left(-\frac23\kappa C_2R^{1/2}p_{n+1}\right),
  \end{align*}
  so that for $n\geq N_0(R)$,
  \begin{align*}
    \frac{1}{t_n}\left\| \int_0^{t_n}D_n(r)\psi-D_{n+1}(r)\psi\dd r \right\|&\leq \Gamma \exp\left(-3 C_2R^{1/2}p_{n+1}\right)
  \end{align*}
  since $\kappa\geq 9/2$.

  Thus, for $\Lambda(\psi,R)=\Lambda:=2(\Gamma+2M_3)$ and $n\geq N_0(R)$, we have
  \begin{align}\label{eq:20231}
    \begin{split}
      \|Q_n\psi-Q_{n+1}\psi\|&\leq \Lambda\exp\left(-3C_2R^{1/2}p_{n+1} \right),
    \end{split}
  \end{align}
  and iterating the inequality $p_{n+1}\geq 2p_n$ in the above, we find
  \begin{align}\label{eq:June11}
  \|Q_n\psi-Q_{n+1}\psi\|&\leq \Lambda\exp\left(-3C_2R^{1/2}2^{n}p_{0} \right)
  \end{align}
  from which \eqref{eq:summable} follows.

  Thus, for each $\psi\in \mathcal{D}_s$, $\lim_{n\to \infty}Q_n\psi$ exists, and we define $Q\psi$ to be the limit. We now show that for these $\psi$, $\lim_{t\to\infty}\frac1tX_H(t)\psi=Q\psi$. Using the triangle inequality, we have for $\psi\in \mathcal{D}_s$ and any $t\in\R$, $n\in \N$,
  \begin{align*}
    \left\|\frac1tX_H(t)\psi-Q\psi \right\|\leq
    \frac1t\|X_H(t)\psi-X_{H_n}\psi \|+\left\|\frac1tX_{H_n}\psi-Q_n\psi \right\|+\|Q_n\psi-Q\psi\|.
  \end{align*}
  Taking $t_n$ as in the above, for $t_{n-1}\leq t<t_n$, we have by Lemma~\ref{lem:posprop}
  \begin{align*}
    \frac1t\|X_H(t)\psi-X_{H_n}\psi \|&\leq
    \Gamma(1/\sqrt{t_{n-1}}+t_{n})\| V-V_n\|_{\infty}^{1/2}
  \end{align*}
  which tends to $0$ as $n\to \infty$ by \eqref{eq:expapprox}, and thus as $t\to \infty$.

  To bound the second term, we may use Corollary~\ref{cor:Corollary} to find 
  \begin{align*}
  \left\|\frac1tX_{H_n}(t)\psi-Q_n\psi \right\| &\leq t_{n-1}^{-1}\|X\psi\|+2C_1M_3p_n^{3/2}\exp\left(\frac12C_2R^{1/2}p_n\right)t_{n-1}^{-1/5}\\
  &=t_{n-1}^{-1}\|X\psi\|+2M_3\exp\left(\frac12(1-2\kappa)C_2R^{1/2}p_n\right),
  \end{align*}
  which, for our choice of $\kappa$, clearly tends to $0$ as $n\to \infty$, and thus as $t\to \infty$. Lastly, by the definition of $Q\psi$, $\|Q_n\psi-Q\psi\|\to 0$ as $n\to \infty$, and we may conclude that
  \[
  \lim_{t\to\infty}\frac1tX_H(t)\psi=Q\psi.
  \]

  Finally, we show that for $\psi\in \mathcal{D}_s$ with $\psi\ne 0$, $\|Q\psi\|>0$. The first step is to find a quantitative lower bound on $Q_n\psi$, using the expression defined in \eqref{eq:2023}. For a fixed $n\in\N$ and $S_n$ defined as in Lemma~\ref{lem:fourierbound}, we compute
  \begin{align*}
  \|Q_n\psi\|^2&=\int_{\T_n^*}\left\|\sum_{m=1}^\infty\frac{\dd E_m(k)}{\dd k}P_m(k)U_{p_n}\psi(k,\cdot)\ \right\|^2\frac{\dd k}{|\T_n^*|}\\
  &\geq \int_{S_n}\left\|\sum_{m=1}^\infty\frac{\dd E_m(k)}{\dd k}P_m(k)U_{p_n}\psi(k,\cdot)\right\|^2\frac{\dd k}{|\T_n^*|}\\
  &=\int_{S_n}\sum_{m=1}^\infty\left|\frac{\dd E_m(k)}{\dd k}\right|^2\|P_m(k)U_{p_n}\psi(k,\cdot)\ \|^2\frac{\dd k}{|\T_n^*|}\\
  &\geq \frac{2\exp(-4C_2R^{1/2}p_n)}{C_2^2p_n^2}\int_{S_n}\sum_{m=1}^\infty\|P_m(k)U_{p_n}\psi(k,\cdot) \|^2\frac{\dd k}{|\T_n^*|}\\
  &=\frac{2\exp(-4C_2R^{1/2}p_n)}{C_2^2p_n^2}\int_{S_n}\|U_{p_n}\psi(k,\cdot)\|^2\frac{\dd k}{|\T_n^*|}
  \end{align*}
  by Lemma~\ref{lem:velocitybound} and Parseval's identity. Using Lemma~\ref{lem:fourierbound}, we have for $n\geq N_1(\psi)$
  \begin{align}
    \|Q_n\psi\|^2
    &\geq \frac{\exp(-4C_2R^{1/2}p_n)}{4C_2^2p_n^2}\|\psi\|^2.
  \end{align}

  Now we find a bound for $\|Q_n\psi-Q\psi\|$. Using \eqref{eq:20231} and $p_{\ell+n}\geq 2^\ell p_{n}$ for each $\ell\in \N$, we may estimate for a given $n\geq N_0(R)$,
  \begin{align*}
  \|Q_n\psi-Q\psi\|&\leq \sum_{\ell=n}^\infty\|Q_{\ell}\psi-Q_{\ell+1}\psi\|\leq \Lambda\sum_{\ell=n}^\infty \exp\left( -3C_2R^{1/2}p_{\ell+1} \right)\\
  &= \Lambda\sum_{\ell=1}^\infty \exp\left( -3C_2R^{1/2}p_{n+\ell} \right)\leq \Lambda\sum_{\ell=1}^\infty \exp\left( -3C_2R^{1/2}2^{\ell}p_{n} \right)\\
  &\leq \Lambda \frac{\exp\left( -3C_2R^{1/2}p_{n}\right)}{1-\exp\left( -3C_2R^{1/2}p_n\right)}
  \end{align*}
  where we have estimated by a much larger geometric series.

  Finally, putting these two estimates together, we have for $n\geq N:=\max\{ N_0,N_1\}$,
  \begin{align*}
  \|Q\psi\|&\geq \|Q_n\psi\|-\|(Q-Q_n)\psi\|\\
  &\geq \frac{\exp\left( -2C_2R^{1/2}p_n\right)}{2C_2p_n}\|\psi\|-\Lambda(\psi,R)\frac{\exp\left( -3C_2R^{1/2}p_{n}\right)}{1-\exp\left( -3C_2R^{1/2}p_n\right)}.
  \end{align*}
  By comparing the exponents in the two terms above, it is clear that we may find an $n\geq N$ large enough such that the right hand side of the above is greater than $0$. Thus, $Q\psi\ne 0$.
  \end{proof}

\section{Appendix}

	The two results of this section are essentially the main results of \cite{RadinSimon}. However, since we must take extra care in the dependencies of the constants which appear in their estimates, we include statements tailored to our setting as well as proofs where necessary. More precisely, Theorems \ref{thm:RadinSimon1} and \ref{thm:RadinSimon2} correspond to Theorems 2.1 and 2.2 of \cite{RadinSimon}, respectively. However, in that paper, the dependencies critical in our work are not recorded. We note also that those authors work in the more general setting of $H_0$ form bounded or $H_0$ operator bounded potentials with relative bound less than $1$, while we deal with bounded potentials.

  \begin{theorem}\label{thm:RadinSimon1}
    Suppose there is an $R>0$ with $\|V\|_\infty\leq R$. Then, for $\psi\in H^1(\R)\cap \dom(X)$,
    \begin{align}\label{eq:radinsimon1}
    \|Xe^{-iHt}\psi\|\leq \alpha(1+|t|)(\|\psi\|_{H^1(\R)}^2+\|X\psi\|^2)^{1/2}
    \end{align}
    for a constant $\alpha(R)=\alpha$.
  \end{theorem}

  \begin{theorem}\label{thm:RadinSimon2}
    Suppose there is an $R>0$ with $\|V\|_\infty\leq R$. Then, for $\psi\in H^2(\R)\cap D(X^2)$,
    \begin{align}\label{eq:radinsimon2}
    \|X^2e^{-iHt}\psi\|\leq \beta(1+t^2)(\|\psi\|_{H^2(\R)}^2+\|X^2\psi\|^2)^{1/2}
    \end{align}
    for a constant $\beta(R)=\beta$.
  \end{theorem}

	Theorem~\ref{thm:RadinSimon1} follows from taking  $a=0$ and $b=R$ in the proof of \cite{RadinSimon}[Theorem 2.1]. Instead, the proof of the statement corresponding to Theorem~\ref{thm:RadinSimon2} in \cite{RadinSimon} is formal. We provide a rigorous proof in our setting below for completeness.

  \begin{proof}[Proof of Theorem~\ref{thm:RadinSimon2}]

  Since we will reduce position estimates to momentum estimates, we begin by recording a pair of estimates of the latter type. For $\psi\in H^2(\R)$,
  \begin{align*}
  \|H_0^{1/2} e^{-itH}\psi\|^2&\leq \|(H_0+V+R)^{1/2} e^{-itH}\psi\|^2\leq \|H_0^{1/2}\psi\|^2+2R\|\psi\|^2
  \end{align*}
  where we have used that
  \begin{align*}
  \langle e^{-itH}\psi,(H_0+V)e^{-itH}\psi\rangle=
  \langle e^{-itH}\psi,e^{-itH}(H_0+V)\psi\rangle=
  \langle \psi,(H_0+V)\psi\rangle.
  \end{align*}
  In particular,
  \begin{align}\label{eq:1june15}
    \|H_0^{1/2} e^{-itH}\psi\|^2\leq (1+2R)\|\psi\|^{2}_{H^1(\R)}.
  \end{align}
  Similarly, we have
  \begin{align*}
  \|H_0e^{-itH}\psi\|\leq \|(H_0+V)e^{-itH}\psi\|+\|Ve^{-itH}\psi\|
  \leq \|H_0 \psi\|+2R\|\psi\|
  \end{align*}
  and
  \begin{align}\label{eq:2june15}
  \|H_0e^{-itH}\psi\|\leq \sqrt{2}(1+2R)\|\psi\|_{H^2(\R)}.
  \end{align}

  We now define $F_\epsilon(y)=\frac{y^4}{1+\epsilon y^4}$ and take $F_\epsilon(X(t)):=e^{iHt}F_\epsilon(X)e^{-iHt}$. Our first task is to compute the strong derivative $\frac{\dd}{\dd t}F_\epsilon(X(t))\psi$, for $\psi\in H^2(\R)$.
  Fix $\psi\in H^2(\R)$, then
  \begin{align*}
    e^{i(t+h)H}F_\epsilon(X)e^{-i(t+h)H}\psi-e^{itH}F_\epsilon(X)e^{-itH}\psi&=
    e^{i(t+h)H}F_\epsilon(X)e^{-i(t+h)H}\psi
    -e^{i(t+h)H}F_\epsilon(X)e^{-itH}\psi\\
    &+e^{i(t+h)H}F_\epsilon(X)e^{-itH}\psi
    -e^{itH}F_\epsilon(X)e^{-itH}\psi.
  \end{align*}
  We examine the difference quotient of each term individually. For the first, using unitarity of $e^{-it H}$, we have
  \begin{align*}
    \|1/h(e^{i(t+h)H}&F_\epsilon(X)e^{-i(t+h)H}
    -e^{i(t+h)H}F_\epsilon(X)e^{-itH})\psi+ie^{itH}F_\epsilon(X)e^{-itH}H\psi\|\\
    &=\|F_\epsilon(X)\frac1h(e^{-i(t+h)H}
    -e^{-itH}\psi)+ie^{-ihH}F_\epsilon(X)e^{-itH}H\psi\|,
  \end{align*}
  which tends to $0$ as $h\to 0$ for $\psi\in H^2(\R)$ by strong continuity of $e^{itH}$ and since $\dom(H)=H^2(\R)$. For the second term, we note that a quick computation shows that $F_\epsilon(X)(H^2(\R))\subseteq H^2(\R)$, and we may compute again
  \begin{align*}
    \|1/h(e^{i(t+h)H}&F_\epsilon(X)e^{-itH}\psi
    -e^{itH}F_\epsilon(X)e^{-itH})\psi)-iHe^{itH}F_\epsilon(X)e^{-itH}\psi\|\\
    &\leq \|1/h(e^{ihH}-I)F_\epsilon(X)\psi-iHF_\epsilon(X)e^{-itH}\psi\|
  \end{align*}
  which again tends to $0$ as $h$ does.
  Thus, as a map from $H^2(\R)$ to $L^2(\R)$, $F_\epsilon(X(t))$ is strongly differentiable with
  \begin{align*}
    \frac{\dd}{\dd t}F_\epsilon(X(t))&=i[H,F_\epsilon(X(t))]=ie^{itH}[H_0,F_\epsilon(X)]e^{-itH}\\
    &=G_\epsilon(X(t))D(t)+D(t)G_\epsilon(X(t))
  \end{align*}
  where $G_\epsilon(y)=\partial_yF_\epsilon(y)=\frac{4y^3}{(1+\epsilon y^4)^2}$, and $G_\epsilon(X(t))=e^{itH}G_\epsilon(X)e^{-itH}$, $D(t)=-e^{itH}i\frac{\dd}{\dd x}e^{-itH}$ as before.

  Fixing $\psi\in H^2(\R)\cap D(X^2)$, we use the above to find
  \begin{align*}
  \frac{\dd}{\dd t}\langle \psi,(F_\epsilon(X(t))+1)\psi\rangle= 2\langle G_\epsilon(X(t))\psi,D(t)\psi\rangle
  \end{align*}
  where we note that
  $G_\epsilon(y)=4F_\epsilon(y)^{1/2}\frac{y}{(1+\epsilon y^4)^{3/2}}$ and so, setting $g_\epsilon(y)=\frac{y^2}{(1+\epsilon y^4)^{3}}$, we have 
  \begin{align}\label{eq:1june16}
  \frac{\dd}{\dd t}\langle \psi,(F_\epsilon(X(t))+1)\psi\rangle\leq  8\langle \psi,F_\epsilon(X(t))\psi\rangle^{1/2}\langle D(t)\psi,g_\epsilon(X(t))D(t)\psi\rangle^{1/2}
  \end{align}
  by the Cauchy-Schwarz inequality.

  We note that since $g_\epsilon\in H^1(\R)$, for $f\in H^1(\R)$, $g_\epsilon(X)Df=Dg_\epsilon(X)f+i(\partial_xg_\epsilon)(X)f$ by the product rule.
  Thus, since $e^{-itH}:H^2(\R)\to H^2(\R)\subset H^1(\R)$, we have
  \begin{align*}
  \langle D(t)\psi,g_\epsilon(X(t))D(t)\psi\rangle=
  \langle D(t)\psi,D(t)g_\epsilon(X(t))\psi\rangle+i\langle D(t)\psi,(\partial_xg_\epsilon)(X(t))\psi\rangle,
  \end{align*}
  and so by the Cauchy-Schwarz inequality again
  \begin{align*}
  \langle D(t)\psi,g_\epsilon(X(t))D(t)\psi\rangle\leq
  \left(\|H_0^{1/2}e^{-iHt}\psi\|^2+\|H_0e^{-iHt}\psi\|^2 \right)^{1/2}\left(\|g_\epsilon(X)e^{-iHt}\psi\|^2+\|(\partial_xg_\epsilon(X))e^{-iHt}\psi\|^2 \right)^{1/2}.
  \end{align*}
  We compute
  \begin{align*}
  \left|\frac{\partial_yg_\epsilon(y)}{2F_\epsilon(y)^{1/4}}\right|&=
  \left|\frac{(1+\epsilon y^4)^{1/4}}{2y}\right|\cdot\left|\frac{2y-10\epsilon y^5}{(1+\epsilon y^4)^4}\right|
  \\
  &=\left|\frac{1-5\epsilon y^4}{(1+\epsilon y^4)^{15/4}}\right|\leq \frac{1+5\epsilon y^4}{1+\epsilon y^4}\\
  &\leq 1+\frac{5\epsilon y^4}{1+\epsilon y^4}\leq 6;
  \end{align*}
  yielding $\partial_yg_\epsilon(y)^2\leq 144F_\epsilon(y)^{1/2}$. Thus,
  \begin{align*}
  \|(\partial_xg_\epsilon(X))e^{-iHt}\psi\|^2\leq 144\langle \psi,F_\epsilon(X(t))^{1/2}\psi\rangle.
  \end{align*}
  We also have $g_\epsilon(y)^2\leq F_\epsilon(y)$, which implies
  \begin{align*}
  \|g_\epsilon(X)e^{-iHt}\psi\|^2\leq \langle \psi, F_\epsilon(X(t))\psi\rangle,
  \end{align*}
  so that
  \begin{align*}
  \|(\partial_xg_\epsilon(X))e^{-iHt}\psi\|^2+\|g_\epsilon(X)e^{-iHt}\psi\|^2\leq 288\langle \psi, (F_\epsilon(X(t))+1)\psi\rangle,
  \end{align*}
  by the elementary inequality $(y+\sqrt{y})\leq 2(1+y)$.

  The inequality \eqref{eq:2june15} and the form bound \eqref{eq:1june15} yields
  \begin{align*}
  \|H_0^{1/2}e^{-iHt}\psi\|^2+\|H_0e^{-iHt}\psi\|^2&\leq (1+2R)\|\psi \|_{H^1(\R)}^2+2(1+2R)^2\|\psi\|_{H^2(\R)}^2\\
  &\leq 4(1+2R)^2\|\psi\|_{H^2(\R)}^2.
  \end{align*}

  Combining the above inequalities to estimate \eqref{eq:1june16}, we have, with $A(R)=2^{19/4}\sqrt{3}(1+2R)^{1/2}$,
  \begin{align*}
  &\frac{\dd}{\dd t}\langle \psi,(F_\epsilon(X(t))+1)\psi\rangle\leq  8\langle \psi,F_\epsilon(X(t))\psi\rangle^{1/2}\langle D(t)\psi,g_\epsilon(X(t))D(t)\psi\rangle^{1/2}\\
  &\leq  A(R)\|\psi\|_{H^2(\R)}^{1/2}\langle \psi,F_\epsilon(X(t))\psi\rangle^{1/2}
  \langle \psi, (F_\epsilon(X(t))+1)\psi\rangle^{1/4}\\
  &\leq A(R)\|\psi\|_{H^2(\R)}^{1/2}
  \langle \psi, (F_\epsilon(X(t))+1)\psi\rangle^{3/4},
  \end{align*}
  and integrating the inequality yields
  \begin{align*}
  \langle \psi,(F_\epsilon(X(t)))\psi\rangle^{1/4}&\leq
  \langle \psi,(F_\epsilon(X(t))+1)\psi\rangle^{1/4}\leq \langle \psi,F_\epsilon(0)\psi\rangle^{1/4}+\frac{A(R)}{4}\|\psi\|_{H^2(\R)}^{1/2}|t|.
  \end{align*}
  Taking $\epsilon\searrow 0$ in the above inequality and using the monotone convergence theorem, we have
  \begin{align*}
  \| X^2e^{-itH}\psi\|^{1/2}\leq \|X^2\psi\|^{1/2}+\frac{A(R)}{4}\|\psi\|_{H^2(\R)}^{1/2}|t|.
  \end{align*}
  Squaring and applying some elementary inequalities, \eqref{eq:radinsimon2} follows.
  \end{proof}

\bibliographystyle{amsplain}
\bibliography{lit}

\end{document}